\documentclass[12pt]{amsart}
\usepackage[utf8]{inputenc}
\usepackage{tikz}
\usepackage{amsthm}
\usepackage{amsmath}
\usepackage{amssymb}
\usepackage{amsthm}
\usepackage{amscd}
\usepackage{mathtools}
\usepackage[all]{xy}
\usepackage[margin=1in]{geometry}
\usepackage{comment}
\usepackage{tikz-cd}
\usepackage{pdfpages}
\usepackage{amscd}
\usepackage{graphicx}
\usepackage{amsthm} 
\usepackage[english]{babel}
\usepackage{hyperref} 

\newtheorem{thm}{Theorem}[section]
\newtheorem{lem}[thm]{Lemma}
\newtheorem{conj}[thm]{Conjecture}
\newtheorem{prop}[thm]{Proposition}

\theoremstyle{definition}

\newtheorem{definition}[thm]{Definition}
\newtheorem{remark}[thm]{Remark}
\newtheorem{example}[thm]{Example}

\numberwithin{equation}{section}

\newcommand{\CC}{\mathbb{C}}

\newcommand\xto[1]{\xrightarrow{#1}}
\newcommand{\mf}{\mathfrak}
\newcommand{\mb}{\mathbb}

\def\Spec{\mbox{\rm Spec}}

\DeclareMathOperator\rank{rank}
\DeclareMathOperator\Gr{Gr}
\DeclareMathOperator\Sym{Sym}
\DeclareMathOperator\codim{codim}
\newcommand{\colorA}[1]{{\color{blue}{#1}}}
\newcommand{\colorB}[1]{{\color{red}{#1}}}
\newcommand{\colorC}[1]{{\color{blue}{#1}}} 

\begin{document}

\title{Free resolutions constructed from bigradings on Lie algebras}

\author{Xianglong Ni}
\address{Department of Mathematics, UC Berkeley, CA 94720}
\email{xlni@berkeley.edu}

\author{Jerzy Weyman}
\address{Instytut Matematyki, Uniwersytet Jagiello\'nski, Krak\'ow, Poland}
\email{jerzy.weyman@uj.edu.pl}

\maketitle

\begin{abstract}
	We construct two families of free resolutions that resolve the ideals of certain opposite Schubert varieties restricted to the big open cell. We conjecture that these examples have genericity properties translating to structure theorems for perfect ideals with given Betti numbers, extending the well-known theorem of Buchsbaum and Eisenbud on Gorenstein ideals of codimension three.
\end{abstract} 

\setcounter{tocdepth}{2}
\tableofcontents

\section{Introduction}

In \cite{BE77}, Buchsbaum and Eisenbud gave their classical result characterizing free resolutions of Gorenstein ideals $I$ of codimension three in a local ring $R$: up to a change of basis in the free modules, the resolution of $R/I$ has the form
\[
	0 \to R \xto{d_3} F_1^* \xto{d_2} F_1 \xto{d_1} R
\]
where $F_1 = R^n$ for an odd integer $n \geq 3$, $d_2$ is a skew-symmetric matrix, and $d_1, d_3$ are comprised of the $(n-1)\times (n-1)$ Pfaffians of $d_2$. In particular, the generic such resolution over the polynomial ring $S$ in variables $x_{ij}$ ($1 \leq i < j \leq n$) is acyclic, and all other such resolutions are obtained via some specialization $S \to R$.

In a sense, providing the generic example explicitly makes this the strongest type of structure theorem possible. Other results of this type include the classical Hilbert-Burch theorem when $\operatorname{pdim}R/I = 2$. Since almost complete intersections are linked to Gorenstein ideals, their resolutions are fully understood as well in codimension three \cite{annebrown}, \cite{CVW21}.

But aside from these cases, such definitive structure theorems remain lacking. Insofar as generic examples go, the closest would be the conjecture---mostly settled, see \cite{HM85} and \cite{VV86}---that Gorenstein ideals of codimension four on six generators are hypersurface sections of Gorenstein ideals of codimension three. The generic resolution of such an ideal would then be the tensor product of the generic Buchsbaum-Eisenbud complex on five generators with the complex $0\to S \xto{\cdot f} S$ for another indeterminate $f$.

In this paper, we construct two families of complexes resolving perfect ideals, which we conjecture have genericity properties akin to the preceding. Specifically, we conjecture that each of our examples specializes to a resolution of any perfect ideal sharing the same Betti numbers. The motivation for this conjecture comes from a deep connection to Weyman's generic ring \cite{W89}, \cite{W18}. Explaining this connection would require extensive background on the generic ring that is not needed for the construction itself, so we defer it to a forthcoming paper to keep the present one focused on examples.

Our examples originate from Lie algebras of type $D_n$ and $E_n$, using bigradings induced by particular pairs of vertices on the Dynkin diagram. In the first family of examples, for the case $D_n$, we recover the generic Buchsbaum-Eisenbud resolutions for Gorenstein ideals of codimension three on $n$ generators, as well as Brown's closely related results on almost complete intersections. By applying the same construction to $E_n$, we obtain resolutions with Betti numbers $(1,5,6,2), (1,6,7,2), (1,5,7,3), (1,7,8,2)$, and $(1,5,8,4)$. 

In the second family of examples, for the case $E_6$, we recover the (conjecturally) generic resolution for Gorenstein ideals of codimension four on six generators. We then naturally conjecture that our analogous examples from $E_7$ and $E_8$ are generic for seven and eight generators respectively.

We give the necessary representation theory background and then outline the general construction of the resolutions in \S\ref{sec:general-construction}. Then we describe explicitly (when feasible) each length three example in \S\ref{sec:length3} and each length four example in \S\ref{sec:Gor4}. In \S\ref{sec:proofres} we prove that our complexes resolve the ideals of certain opposite Schubert varieties restricted to the big open cell. Finally, in \S\ref{sec:invariants}, we describe how linear sections of the resolutions from \S\ref{sec:length3} and \S\ref{sec:Gor4} can be interpreted in terms of invariants, extending observations made in \cite{SW21}.

\clearpage
\section{The general construction}\label{sec:general-construction}

We will be concerned with simply-laced Dynkin diagrams (types $A_n$, $D_n$, and $E_n$). Although it is nonstandard, we will display these diagrams in the following $T$-shaped manner, so as to match the notation in \cite{W18}.
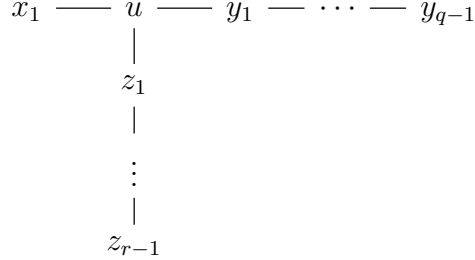
\[\begin{tikzcd}[column sep = small, row sep = small]
x_1 \ar[r,dash] & u \ar[r,dash]\ar[d,dash] & y_1 \ar[r,dash] & \cdots \ar[r,dash] & y_{q-1} \\
& z_1 \ar[d,dash]\\
& \vdots \ar[d,dash]\\
& z_{r-1}
\end{tikzcd}\]
In some degenerate cases, the right arm may have length zero, resulting in type $A_n$. We will primarily be interested in the types $D_n$ and $E_n$, as the interesting examples are produced from those cases.

To such a diagram, let $\mf{g}$ be the associated complex simple Lie algebra. Each node $t$ of the Dynkin diagram corresponds to a simple root $\alpha_t$ of $\mf{g}$. Let $n$ denote the rank of $\mf{g}$, i.e. the number of nodes in the Dynkin diagram. Since the simple roots form a basis of the root space of $\mf{g}$, each finite dimensional representation $V$ of $\mf{g}$ is $\mb{Z}^n$-multigraded. Of course, this includes the adjoint representation: the case of $\mf{g}$ itself.

By designating a particular subset $I$ of the nodes on the Dynkin diagram (equivalently, of the simple roots) we can coarsen this $\mb{Z}^n$ grading to a $\mb{Z}^{|I|}$ grading, which we refer to as the $I$-grading on $V$. Each graded piece is a representation of the subalgebra $\mf{g}^{(I)}$ corresponding to the diagram which remains after deleting the vertices in $I$. For our purposes, $I$ will typically be a singleton set.

For each vertex $t$ of the Dynkin diagram, let $\omega_t$ denote the corresponding fundamental weight. The dominant integral weights $\omega$ are nonnegative integral combinations of the $\omega_t$. Let $V(\omega)$ denote the finite dimensional irreducible representation of $\mf{g}$ with highest weight $\omega$. To refer to a representation of a subalgebra $\mf{g}^{(I)}$, we will write $V(\omega,\mf{g}^{(I)})$ instead.

The full $I$-graded decomposition of $V(\omega)$ can be computed algorithmically; see for example \cite{LW19}. The top piece in particular is very easy to read off: it is simply $V(\omega,\mf{g}^{(I)})$ (ignoring the coefficients of $\omega_t$ in $\omega$ for $t \in I$). We write $i^\mathrm{top}_t$ for the inclusion of the top $t$-graded piece and $p^\mathrm{bottom}_t$ for the projection onto the bottom $t$-graded piece.

\begin{example}
	Let $I = \{x_1\}$. The subalgebra of $\mf{g}$ corresponding to the nodes
	\[
		y_{q-1},\ldots,y_1,u,z_1,\ldots,z_{r-1}
	\]
	is $\mf{sl}(F)$ where $F=\mb{C}^{q+r}$. If $V$ is a representation of $\mf{g}$, then each $x_1$-graded component of $V$ will be a representation of $\mf{sl}(F)$. For example, the top $x_1$-graded component of $V(\omega_{z_{r-1}})$ is $V(\omega_{z_{r-1}},\mf{sl}(F)) = F^*$. For $V(\omega_{y_{q-1}})$ it is $V(\omega_{y_{q-1}},\mf{sl}(F)) = F$, and for $V(\omega_{x_1})$ it is the trivial representation $V(0,\mf{sl}(F)) = \mb{C}$ because $x_1 \in I$.
\end{example}
The bottom components can then be inferred from duality on a case-by-case basis.
\begin{remark}\label{rem:dual-node}
	To keep the exposition in this section simple and uniform, we will assume that the Dynkin type under consideration has self-dual representations. This is the case for $D_n$ (even $n$), $E_7$, and $E_8$. For the remaining cases---i.e. $A_n$ (any $n$), $D_n$ (odd $n$), and $E_6$---replace the vertex $x_1$ with its dual in what follows, and adjust the definition of $\mf{sl}(F_2) = \mf{g}^{(x_1)}$ accordingly. This adjustment is needed in order to still produce a complex and will be explained more precisely when the need arises. See for example \S\ref{subsec:1nn1-n-odd}.
\end{remark}
We now describe the construction of our two families of complexes. For the length three construction, we will be concerned with the gradings induced by the vertices $x_1$ and $z_1$. To better describe the graded components, we fix the following subalgebras of $\mf{g}$:
\begin{itemize}
	\item $\mf{sl}(F_3)$ corresponding to the nodes $z_2,z_3,\ldots,z_{r-1}$, where $F_3 = \mb{C}^{r-1}$,
	\item $\mf{sl}(F_2)$ corresponding to the nodes $y_{q-1},\ldots,y_1,u,z_1,\ldots,z_{r-1}$, where $F_2 = \mb{C}^{q+r}$, and
	\item $\mf{sl}(F_1)$ corresponding to the nodes $y_{q-1},\ldots,y_1,u,x_1$, where $F_1 = \mb{C}^{q+2}$.
\end{itemize}
With this notation, each $x_1$-graded component is a representation of $\mf{sl}(F_2)$ and each $z_1$-graded component is a representation of $\mf{sl}(F_3)\times \mf{sl}(F_1)$. For example, the top $z_1$-graded component of $V(\omega_{z_{r-1}})$ is $F_3^*$, and the bottom $x_1$-graded component is $F_2$ (for the latter we use the assumption that the representation is self-dual).

Let $G$ be the simply-connected complex Lie group corresponding to $\mf{g}$. Then representations $V$ of $\mf{g}$ are also representations of $G$. Let $\rho_V(g)$ denote the action of $g\in G$ on $V$. For each $g \in G$, we construct a complex
\[
	0 \to F_3^* \to F_2 \to F_1^* \to \mb{C}
\]
of $\mb{C}$-vector spaces as follows.

Using the representation $V=V(\omega_{z_{r-1}})$, we define $d_3$ to be the composite
\[
	F_3^* \xto{i^\mathrm{top}_{z_1}} V \xto{\rho_V(g)} V \xto{p^\mathrm{bottom}_{x_1}} F_2.
\]
Using the representation $V=V(\omega_{y_{q-1}})$, we define $d_2$ to be the dual of the composite
\[
	F_1 \xto{i^\mathrm{top}_{z_1}} V \xto{\rho_V(g)} V \xto{p^\mathrm{bottom}_{x_1}} F_2^*.
\]
\begin{remark}\label{rem:dual-rep-inverse}
	Note that the action of $g$ on the dual $V^*$ is precomposition by the action of $g^{-1}$ on $V$. This means that an alternative definition of $d_2$ is the composite
	\[
	F_2 \xto{i^\mathrm{top}_{x_1}} V^* \xto{\rho_V(g^{-1})} V^* \xto{p^\mathrm{bottom}_{z_1}} F_1^*.
	\]
\end{remark}
Lastly, using the representation $V=V(\omega_{x_1})$, we define $d_1$ to be the composite
\[
	F_1^* \xto{i^\mathrm{top}_{z_1}} V \xto{\rho_V(g)} V \xto{p^\mathrm{bottom}_{x_1}} \mb{C}.
\]
Since we identify top and bottom graded components with particular representations of $\mf{sl}(F_i)$, each differential is only defined up to overall multiplication by a nonzero scalar. When we describe these differentials explicitly in later sections, a choice of such a scalar will be implicit.
\begin{lem}
	The sequence of maps defined above forms a complex.
\end{lem}
\begin{proof}
	The composite $d_2 d_3$ is adjoint to the tensor product $d_3 \otimes d_2^*$ followed by the trace:
	\[
		F_3^* \otimes F_1 \to V(\omega_{z_{r-1}}) \otimes V(\omega_{y_{q-1}}) \to V(\omega_{z_{r-1}}) \otimes V(\omega_{y_{q-1}}) \to F_2 \otimes F_2^* \to \mb{C}.
	\]
	The component $F_3^* \otimes F_1$ is the top $z_1$-graded piece of the irreducible representation $V(\omega_{z_{r-1}} + \omega_{y_{q-1}})$ which occurs once inside the tensor product. The bottom $x_1$-graded piece of this irreducible subrepresentation is $V(\omega_{z_{r-1}} + \omega_{y_{q-1}}, \mf{sl}(F_2)) = \mf{sl}(F_2)$, i.e. the complementary factor to $\mb{C}$ inside of $F_2^* \otimes F_2$. Since the components $F_3^* \otimes F_1$ and $\mb{C}$ reside in distinct subrepresentations, this composite is necessarily zero.
	
	A similar argument works for the composite $d_1 d_2$, instead looking at the tensor product $d_2^* \otimes d_1$:
	\[
		\mb{C} \to F_1 \otimes F_1^* \to V(\omega_{y_{q-1}}) \otimes V(\omega_{x_1}) \to V(\omega_{y_{q-1}}) \otimes V(\omega_{x_1}) \to F_2^* \otimes \mb{C}.
	\]
	Here $F_2^* \otimes \mb{C}$ is the bottom $x_1$-graded piece of $V(\omega_{y_{q-1}} + \omega_{x_1})$, which occurs once inside the tensor product. But its top $z_1$-graded piece is $\mf{sl}(F_1)$, complementary to $\mb{C}$ inside of $F_1 \otimes F_1^*$, so once again we deduce that the composite is zero.
\end{proof}

For $t \in \{x_1,z_1\}$, let $\mf{n}_t$ denote the nilpotent subalgebra that is the negative part of $\mf{g}$ in the $t$-grading. The exponential map $\exp \colon \mf{n}_t \to G$ is a diffeomorphism onto its image $N_t$, a unipotent subgroup of $G$. So the coordinate ring of $N_t$ is just that of $\mf{n}_t$, i.e. $R_t = \operatorname{Sym}_\mb{C}(\mf{n}_t^*)$.

Let $L_t \in \mf{n}_t \otimes R_t$ be the ``trace element,'' i.e. the image of $1 \in \mb{C}$ under the map
\[
	\mb{C} \to \mf{n}_t \otimes \mf{n}_t^* \subset \mf{n}_t \otimes R_t.
\]
In other words, $L_t$ represents a generic element of $\mf{n}_t$. The action of $\mf{n}_t$ on each representation $V$ can then be understood as the image of $L_t$ under $\sigma_V \colon (\mf{n}_t \to \operatorname{End}(V)) \otimes R_t$. Note that $\sigma_V(L_t)$ is nilpotent, hence its exponential is algebraically well-defined.

If we replace $\rho_V(g)$ in our construction by $\exp \sigma_V(L_t)$, we obtain a complex $\mb{F}_t$ of free $R_t$-modules
\[
	0 \to F_3^* \otimes R_t \to F_2 \otimes R_t \to F_1^* \otimes R_t \to \mb{C} \otimes R_t.
\]
Note that the polynomial ring $\operatorname{Sym}_\mb{C}(\mf{n}_t^*)$ and the complex $\mb{F}_t$ are $\mb{Z}^n$-multigraded by the simple roots. The multigrading on $\mb{F}_t$ can be deduced completely combinatorially without actually computing the differentials; we illustrate this process for $E_6$ in \S\ref{sec:length3}.

From the above, we actually produce \emph{two} complexes $\mb{F}_{x_1}$ and $\mb{F}_{z_1}$ over different polynomial rings. It is then natural to ask how these compare.
\begin{lem}\label{lem:dual-format-equiv}
	Let $\mf{n} = \mf{n}_{x_1} + \mf{n}_{z_1} \subset \mf{g}$, and $R = \operatorname{Sym}_\mb{C}(\mf{n}^*)$. Viewed over $R$, the complexes $\mb{F}_{x_1}$ and $\mb{F}_{z_1}$ are equivalent in the sense that there exists an automorphism $\phi$ of $R$ such that $\phi^* \mb{F}_{x_1} \cong \mb{F}_{z_1}$ as complexes.
\end{lem}
\begin{proof}
	We introduce a third complex using only the coordinates ``shared'' between $\mb{F}_{x_1}$ and $\mb{F}_{z_1}$. Let $\mf{n}_{x_1,z_1}$ be the part of $\mf{g}$ that is strictly negative in the $x_1,z_1$-bigrading, i.e. $\mf{n}_{x_1,z_1} = \mf{n}_{x_1} \cap \mf{n}_{z_1}$. Let $\mb{F}_{x_1,z_1}$ be the complex obtained using $\mf{n}_{x_1,z_1}$ in place of $\mf{n}_t$ in the construction above. We claim that $\mb{F}_{x_1}$ and $\mb{F}_{z_1}$ are both equivalent to this complex (viewed over $R$), which will suffice to prove the lemma.
	
	For brevity we consider only $\mb{F}_{x_1}$, as the situation for $\mb{F}_{z_1}$ is very similar. Decompose $L_{x_1} = L_0 + L_-$ where $L_0$ is degree 0 with respect to the $z_1$-grading and $L_-$ is concentrated in negative $z_1$-degrees. The Baker-Campbell-Hausdorff formula implies that there exists an $\widetilde{L}$ such that
	\[
		\exp \widetilde{L} = (\exp L_0)(\exp L_-).
	\]
	Explicitly, $\widetilde{L} = L_{x_1} + \frac{1}{2}[L_0,L_-] + \cdots$, and this is well-defined because the terms eventually become zero. Each coordinate of $L_{x_1}$ is a variable $x \in R$, and the formula for $\widetilde{L}$ shows that the corresponding coordinate $x'$ of $\widetilde{L}$ is $x$ plus other terms with strictly lower $z_1$-degree. Thus the map $\phi$ sending each $x$ to $x'$ is invertible. The tensor product $\phi^* \mb{F}_{x_1}$ is just the complex obtained by replacing $L_{x_1}$ with $\widetilde{L}$ throughout the construction.
	
	So, the differentials $d_3, d_2^*, d_1$ of $\phi^* \mb{F}_{x_1}$ have the form
	\[
		p^\mathrm{bottom}_{z_1} \exp(\sigma_V(\widetilde{L})) i^\mathrm{top}_{x_1}
		= p^\mathrm{bottom}_{z_1} \exp(\sigma_V(L_0))\exp(\sigma_V(L_-)) i^\mathrm{top}_{x_1}
		= M_V p^\mathrm{bottom}_{z_1} \exp(\sigma_V(L_-)) i^\mathrm{top}_{x_1}
	\]
	where $M_V$ is the restriction of $\exp(\sigma_V(L_0))$ to the bottom $z_1$-graded piece of $V$, i.e. an automorphism of that graded component. Hence up to change of basis of the modules in the complex, we may ignore the term $M_V$. But then we are just left with $\mb{F}_{x_1,x_2}$, as desired.
\end{proof}
In \S\ref{sec:proofres} we will prove that the complexes $\mb{F}_t$ are in fact acyclic. Our interest in these complexes stems primarily from the following conjecture, c.f. \cite{WICERM}.
\begin{conj}[Genericity conjecture for codimension three]\label{conj:gen3}
	Let $S$ be a local $\mb{C}$-algebra and suppose that
	\[
		\mb{G} \colon 0 \to S^{r-1} \to S^{q+r} \to S^{q+2} \to S
	\]
	resolves $S/I$ for some perfect ideal $I$. Then there exists a homomorphism $\phi\colon R_t \to S$ such that $\phi^* \mb{F}_t \cong \mb{G}$.
\end{conj}
If proven, this conjecture would yield structure theorems for perfect ideals of codimension three with Betti numbers $(1,n,n,1)$, $(1,4,n,n-3)$, $(1,5,6,2)$, $(1,5,7,3)$, $(1,6,7,2)$, $(1,5,8,4)$, and $(1,7,8,2)$. We will see in \S\ref{sec:length3} that this conjecture is consistent with the well-known results concerning the first two cases on this list.

The second family of examples follows a very similar construction. We assume the Dynkin diagram has the form
\[\begin{tikzcd}[column sep = small, row sep = small]
x_1 \ar[r,dash] & u \ar[r,dash]\ar[d,dash] & y_1 \ar[r,dash] & \cdots \ar[r,dash] & y_{q-1} \\
& z_1 \ar[d,dash]\\
& z_2
\end{tikzcd}\]
i.e. it is $E_n$ for $4 \leq n \leq 8$, where $E_4 = A_4$ and $E_5 = D_5$. Instead of using the $x_1$ and $z_1$ gradings, we use the $x_1$ and $z_2$ gradings. (Again Remark~\ref{rem:dual-node} applies.) In order to describe the graded components, we fix the following subalgebras of $\mf{g}$:
\begin{itemize}
	\item $\mf{sl}(F)$ corresponding to the nodes $y_{q-1},\ldots,y_1,u,z_1,z_2$, where $F = \mb{C}^{q+2}$, and
	\item $\mf{sl}(H)$ corresponding to the nodes $y_{q-1},\ldots,y_1,u,z_1$, where $H = \mb{C}^{q+1}$.
\end{itemize}
With this setup, the $x_1$-graded components are representations of $\mf{sl}(F)$ and the $z_2$-graded components are representations of $\mf{so}(H \oplus H^*)$. 

For every $g \in G$ we can construct a self-dual complex of $\mb{C}$-vector spaces
\[
	0 \to \mb{C} \xto{d_4} F^* \xto{d_3} H\oplus H^* \xto{d_2} F \xto{d_1} \mb{C}.
\]
Using the representation $V = V(\omega_{z_2})$, we define $d_4$ to be the composite
\[
	\mb{C} \xto{i^\mathrm{top}_{z_2}} V \xto{\rho_V(g)} V \xto{p^\mathrm{bottom}_{x_1}} F^*.
\]
Using the representation $V = V(\omega_{y_{q-1}})$, we define $d_3$ to be the dual of the composite
\[
H \oplus H^* \xto{i^\mathrm{top}_{z_2}} V \xto{\rho_V(g)} V \xto{p^\mathrm{bottom}_{x_1}} F.
\]
Then we define $d_2 = d_3^*$ and $d_1 = d_4^*$.
\begin{lem}
	The sequence of maps defined above forms a complex.
\end{lem}
\begin{proof}
	This proceeds similarly to the length three case. For the composite $d_3 d_4$, consider the tensor product representation $V(\omega_{z_2}) \otimes V(\omega_{y_{q-1}})$. The irreducible representation $V(\omega_{z_2} + \omega_{y_{q-1}})$ appears inside this tensor product once. Its top $z_2$-graded component is $H \oplus H^*$, but its bottom $x_1$-graded component is $\mf{sl}(F)$, complementary to $\mb{C}$ inside of $F^* \otimes F$. Thus $H\oplus H^*$ and $\mb{C}$ reside in different subrepresentations, implying that the composite is zero.
	
	It remains to consider the composite $d_2 d_3$, which is adjoint to
	\[
		\mb{C} \to S_2(H \oplus H^*) \to S_2 V(\omega_{y_{q-1}}) \to S_2 V(\omega_{y_{q-1}}) \to S_2 F \subset F \otimes F
	\]
	The symmetric square $S_2 (H \oplus H^*)$ decomposes as $\mb{C} \oplus V(2\omega_{y_{q-1}}, \mf{sl}(H \oplus H^*))$, where the former $\mb{C}$ factor is the one of interest. However, the latter piece is the one appearing in $V(2\omega_{y_{q-1}}) \subset S_2 V(\omega_{y_{q-1}})$, whose bottom $z_1$-graded component is $S_2 F$. So again, we see that $\mb{C}$ and $F\otimes F$ reside in different subrepresentations and conclude the composite is zero.
\end{proof}

For $t \in \{x_1,z_2\}$, let $\mf{n}_t$ denote the nilpotent subalgebra that is the negative part of $\mf{g}$ in the $t$-grading. Proceeding exactly as we did for the length three construction, we obtain two complexes $\mb{F}_t$ over polynomial rings $R_t$ using our parametrization of $\exp \mf{n}_t = N_t \subset G$. The statement and proof of Lemma~\ref{lem:dual-format-equiv} apply in this setting as well. Again, we defer proving their acyclicity until \S\ref{sec:proofres}. We have the following conjecture, c.f. \cite{WGor4}.

\begin{conj}[Genericity conjecture for codimension four Gorenstein]\label{conj:gen4}
	Let $S$ be a local $\mb{C}$-algebra and suppose that
	\[
		\mb{G}\colon 0 \to S \to S^{q+2} \to S^{2q+2} \to S^{q+2} \to S
	\]
	resolves $S/I$ for some perfect ideal $I$. Then there exists a homomorphism $\phi\colon R_t \to S$ such that $\phi^* \mb{F}_t \cong \mb{G}$.
\end{conj}
If proven, this conjecture would yield structure theorems for Gorenstein ideals of codimension four on up to $n=8$ generators. As we will see in \S\ref{sec:Gor4}, this conjecture certainly holds for $n \leq 5$, and for $n = 6$ it coincides with a well-known (and mostly settled) conjecture.

\section{Perfect ideals of codimension three}\label{sec:length3}
We now describe the results of the construction explicitly for each of the Dynkin cases. In this section we consider the first family of examples, which yield (conjecturally generic) resolutions of length three resolving perfect ideals. The following resolutions have been fully implemented in Macaulay2 for the formats associated to $D_n$, $E_6$, and $E_7$. For $E_8$ unfortunately, while the adjoint $V(\omega_8)$ is manageable, the representations $V(\omega_2)$ and $V(\omega_1)$ are enormous and so we have only been able to implement one of the three differentials.

The simplest example is actually $A_n$, i.e. when the right arm of the $T$-shaped graph has length zero. This case is not particularly interesting, and there are many others to discuss, so we omit it for brevity. The interested reader can check that Conjecture~\ref{conj:gen3} is evidently correct in this case: the complex obtained resolves a complete intersection on three generators.

For $D_n$, we go through the construction of \S2 in great detail and show how it recovers familiar complexes. When we move on to $E_n$, we shift our focus to describing the final output as explicitly as possible, with the understanding that the construction itself is completely analogous.

As mentioned in \S\ref{sec:general-construction}, we obtain two equivalent complexes $\mb{F}_{x_1}$ and $\mb{F}_{z_1}$ for each format. The complex $\mb{F}_{x_1}$ is simpler to describe explicitly, so we will focus on it.

\subsection{$D_n$}

Here we consider the Lie algebra $\mf{g} = \mf{so}(2n)$ associated to the Dynkin diagram $D_n$. The three fundamental representations of interest are the standard representation $V(\omega_1)$ and the half-spinor representations $V(\omega_{n-1})$ and $V(\omega_n)$. If $n$ is even, then each of these representations is self-dual. However, if $n$ is odd, then the half-spinor representations $V(\omega_{n-1})$ and $V(\omega_n)$ are dual to each other. We will consider the case of $n$ even first, and then briefly describe the changes necessary for the case of $n$ odd afterwards.

\subsubsection{Vertices $(n,n-1)$, format $(1,n,n,1)$, $n$ even}\label{subsec:1nn1-n-even}
For this format, we draw the Dynkin diagram as
\[\begin{tikzcd}
\colorA{n} \ar[r,dash]& n-2 \ar[r,dash] & n-3 \ar[r,dash] & \cdots \ar[r,dash] & 2 \ar[r,dash] & 1\\
& \colorB{n-1} \ar[u,dash]
\end{tikzcd}\]
to match the ``$T_{p,q,r}$'' convention from e.g. \cite{W18}. The subalgebras $\mf{sl}(F_1)$ and $\mf{sl}(F_2)$ of $\mf{g}$ are as in \S\ref{sec:general-construction}, with $F_1 = F_2 = \mb{C}^n$.
Note that $\mf{g} = \mf{so}(F_1 \oplus F_1^*)$ (or $\mf{so}(F_2 \oplus F_2^*)$).
The decompositions of $V(\omega_{n-1})$, $V(\omega_1)$, and $V(\omega_n)$ in the gradings induced by the simple roots $\alpha_{n-1}$ and $\alpha_n$ are displayed below. The boxed components will constitute the modules in our free resolution.

\begin{align*}
V(\omega_{n-1}) &= \mb{C} \oplus \bigwedge^2 F_1 \oplus \cdots \oplus \colorB{\boxed{\bigwedge^n F_1}}\\
&= \colorA{\boxed{F_2}} \oplus \bigwedge^3 F_2 \oplus \cdots \oplus \bigwedge^{n-1} F_2\\
V(\omega_1) &= F_1^* \oplus \colorB{\boxed{F_1}}\\
&= \colorA{\boxed{F_2^*}} \oplus F_2\\
V(\omega_n) &= F_1 \oplus \bigwedge^3 F_1 \oplus \cdots \oplus \colorB{\boxed{\bigwedge^{n-1} F_1}}\\
&= \colorA{\boxed{\mb{C}}} \oplus \bigwedge^2 F_2 \oplus \cdots \oplus \bigwedge^n F_2
\end{align*}

Let $\mf{n}$ be the negative part of $\mf{g}$ in the $(n-1)$-grading\footnote{There is also the option of taking the negative part of $\mf{g}$ in the $n$-grading instead, but the symmetry of the nodes $n$ and $n-1$ on $D_n$ makes it redundant. The complex obtained in that case is dual to the one presented here.}, i.e.
\[
	\mf{n} = \bigwedge^2 F_1^*,\\
\]
and let $R = \operatorname{Sym}_\mb{C}(\mf{n}^*)$ be its ($\mb{Z}^n$-graded) coordinate ring. (From the perspective of the coarser $(n-1)$-grading, this is an ordinary polynomial ring with all generators in degree 1.)

To be more explicit, we will pick bases according to the following convention.
\begin{definition}\label{def:std-basis}
	If $\mf{sl}(F)$ corresponds to the nodes $t_1,\ldots,t_{n-1}$ on the Dynkin diagram, we will say the ordered list $e_1,\ldots,e_n$ is a \emph{standard basis} of $F$ if $e_1$ is a highest weight vector and $e_2,\ldots,e_n$ are obtained from $e_1$ by sequentially applying the Lie algebra generators corresponding to the roots $-\alpha_{t_1},\ldots,-\alpha_{t_{n-1}}$.
\end{definition}

Let $e_1,\ldots,e_n$ be a standard basis of $F_1$. Then
\[
	R = \mb{C}[\{b_{ij}\}_{1 \leq i < j \leq n}]
\]
where $b_{ij}$ corresponds to $e_i \wedge e_j$.

We will construct a resolution of the form
\[
0 \to \colorB{\mb{C}} \otimes R \xto{\exp_{n-1}} \colorA{F_2} \otimes R \xto{\exp_1^*} \colorB{F_1^*}\otimes R \xto{\exp_n} \colorA{\CC}\otimes R
\]
where $\exp_i$ means the action of $\exp \mathfrak{n}$ on $V(\omega_i)$, which we explain shortly.

To describe the resolution, we begin with the differential $d_2\colon F_2 \otimes R \to F_1^* \otimes R$ since it comes from the smallest representation of the three: the standard representation $V(\omega_1)$. In the $(n-1)$-grading, this representation has only two graded components: $F_1^*$ and $F_1$. The action of $\mf{n}$ on $V(\omega_1)$ is the evident contraction map
\[
	(\bigwedge^2 F_1^*) \otimes F_1 \to F_1^*.
\]
That is, an element of $\mf{n}$ is equivalently an alternating map $F_1 \to F_1^*$, and its action on $V(\omega_1)$ is precisely that map. With our coordinates $b_{ij}$ on $\mf{n}$, the action of a generic element of $\mf{n}$ is the strictly upper-triangular block matrix
\[
	X = \begin{bmatrix}
	0_{n\times n} & B \\
	0_{n\times n} & 0_{n\times n}
	\end{bmatrix}
\]
where $B \colon F_1 \to F_1^*$ is the generic skew matrix on the variables $b_{ij}$ and $0_{n\times n}$ means a $n\times n$ zero matrix. The exponential of this action is then
\[
	\exp X = \begin{bmatrix}
	I_n & B\\
	0_{n\times n} & I_n
	\end{bmatrix}
\]
As per \S2, we take $d_2^*$ to be the part of this matrix mapping from the top piece in the $(n-1)$-grading, $F_1$, to the bottom piece in the $n$-grading, $F_2^*$. Writing $\epsilon_1,\ldots,\epsilon_n$ for the dual basis of $e_1,\ldots,e_n$, a basis of $F_2^*$ in $V(\omega_1)$ is given by $\epsilon_1,\ldots,\epsilon_{n-1} \in F_1^*$ together with $e_n \in F_1$.

Thus
\[
	d_2 = \begin{bmatrix}
	0 & -b_{12} & \cdots & -b_{1(n-1)} & 0\\
	b_{12} & 0 & \cdots & -b_{2(n-1)} & 0\\
	\vdots & \vdots & \ddots & \vdots & \vdots \\
	b_{1(n-1)} & b_{2(n-1)} & \cdots & 0 & 0\\
	b_{1n} & b_{2n} & \cdots & b_{(n-1)n} & -1
	\end{bmatrix}
\]
\begin{remark}
	Many steps of the construction involve identifying representations explicitly. Such isomorphisms can be adjusted by any nonzero scalar, so there is a considerable amount of freedom in signs and coefficients throughout.
\end{remark}
For the differential $d_3$, we turn to the representation $V(\omega_{n-1})$. The action of $\mf{n}$ on $V(\omega_{n-1})$ is given by the contractions
\[
	\bigwedge^2 F_1^* \otimes \bigwedge^{k} F_1 \to \bigwedge^{k-2} F_1.
\]
In terms of the coordinates $b_{ij}$ on $\mf{n}$, this action sends $e_1 \wedge \cdots \wedge e_k \in \bigwedge^k F_1$ to
\begin{equation}\label{eq:halfspinoraction1}
	\sum_{1 \leq i < j \leq k} (-1)^{i+j} b_{ij} \Big(\bigwedge_{\substack{1 \leq r \leq k\\r \neq i,j}} e_r\Big) \in \bigwedge^{k-2} F_1
\end{equation}
and analogously for other basis elements of $\bigwedge^k F_1$. Denote this action by $X$. It is straightforward to show that
\begin{equation}\label{eq:halfspinoraction2}
	X^s(e_1\wedge \cdots \wedge e_n) = s! \sum_{\substack{I \subseteq \{1,\ldots,n\}\\|I|=2s}} \pm P_I \Big(\bigwedge_{\substack{1 \leq r \leq k\\r \notin I}} e_r\Big) \in \bigwedge^{n - 2s} F_1
\end{equation}
where $P_I$ denotes the Pfaffian of the submatrix of $B$ with row and column indices from $I$.

The differential $d_3$ is the part of $\exp X$ mapping from the top part of $V(\omega_{n-1})$ in the $(n-1)$-grading, $\bigwedge^n F_1$, to the bottom part in the $n$-grading, $F_2$. The latter, from the perspective of the $(n-1)$-grading, has basis given by $e_i \wedge e_n \in \bigwedge^2 F_1$ ($1 \leq i \leq n-1$) together with $1 \in \mb{C}$. Thus our differential is
\[
	d_3 = \begin{bmatrix}
	P_{\{2,3,4,\ldots,n-2,n-1\}}\\
	-P_{\{1,3,4,\ldots,n-2,n-1\}}\\
	P_{\{1,2,4,\ldots,n-2,n-1\}}\\
	\vdots\\
	-P_{\{1,2,3,\ldots,n-3,n-1\}}\\
	P_{\{1,2,3,\ldots,n-3,n-2\}}\\
	P_{\{1,2,3,\ldots,n-3,n-2,n-1,n\}}
	\end{bmatrix}
\]
Note that the first $n-1$ entries are $(n-2)\times(n-2)$ sub-Pfaffians of $B$, while the last entry is the Pfaffian of the whole matrix $B$.

Finally, there is the differential $d_1$ coming from the other half-spinor representation $V(\omega_n)$. As the situation is similar to that of $d_3$, we will be brief and just state the answer:
\[
	d_1 = \begin{bmatrix}
	-P_{\{2,3,4,\ldots,n-2,n-1\}} &
	P_{\{1,2,4,\ldots,n-2,n-1\}} &
	\cdots &
	P_{\{1,2,3,\ldots,n-3,n-1\}} &
	-P_{\{1,2,3,\ldots,n-3,n-2\}} & 0
	\end{bmatrix}
\]
The resolution we've constructed is evidently non-minimal. But in view of Conjecture~\ref{conj:gen3}, this is to be expected, as Gorenstein ideals of codimension three are minimally generated by an \emph{odd} number of elements, whereas $n$ here is even \cite{Watanabe}.

\begin{remark}\label{rem:1nn1-even-odd-comparison}
	This resolution was also produced in \cite[Theorem 4.1]{GW20}, though by different means. Moreover, by deleting the last zero entry of $d_1$, the last column of $d_2$, and the last entry of $d_3$ (the Pfaffian of $B$), we obtain an equivalent minimal resolution of format $(1,n-1,n-1,1)$ which is none other than the familiar generic example from \cite{BE77}. This minimal resolution is also the output of our construction when applied in the case of $n$ odd in \S\ref{subsec:1nn1-n-odd}.
\end{remark}

\subsubsection{Vertices $(n,n-1)$, format $(1,n,n,1)$, $n$ odd}\label{subsec:1nn1-n-odd}
We use the same notation as in the preceding subsection, except now we assume that $n$ is odd.

For this case, the decompositions of $V(\omega_{n})$, $V(\omega_1)$, and $V(\omega_{n-1})$ in the gradings induced by the simple roots $\alpha_{n-1}$ and $\alpha_n$ are displayed below.
\begin{align*}
V(\omega_{n-1}) &= F_1 \oplus \bigwedge^3 F_1 \oplus \cdots \oplus \colorB{\boxed{\bigwedge^{n} F_1}}\\
&= \colorA{\boxed{\mb{C}}} \oplus \bigwedge^2 F_2 \oplus \cdots \oplus \bigwedge^{n-1} F_2\\
V(\omega_1) &= F_1^* \oplus \colorB{\boxed{F_1}}\\
&= \colorA{\boxed{F_2^*}} \oplus F_2\\
V(\omega_{n}) &= \mb{C} \oplus \bigwedge^2 F_1 \oplus \cdots \oplus \colorB{\boxed{\bigwedge^{n-1} F_1}}\\
&= \colorA{\boxed{F_2}} \oplus \bigwedge^3 F_2 \oplus \cdots \oplus \bigwedge^{n} F_2
\end{align*}
Here we run into an obstacle, which is that the boxed components above no longer match up as they did for the case of $n$ even. The construction before would yield two $n\times n$ matrices from $V(\omega_{n})$ and $V(\omega_1)$, and a $1\times 1$ matrix from $V(\omega_{n-1})$, which evidently do not piece together into a $(1,n,n,1)$ resolution.

The adjustment mentioned in Remark~\ref{rem:dual-node} remedies this. The vertices $n$ and $n-1$ are dual on the Dynkin diagram. If instead of the bottom pieces in the $n$-grading, we use the bottom pieces in the $(n-1)$-grading, we get
\begin{align*}
V(\omega_{n-1}) &= \colorA{\boxed{F_1}} \oplus \bigwedge^3 F_1 \oplus \cdots \oplus \colorB{\boxed{\bigwedge^{n} F_1}}\\
V(\omega_1) &= \colorA{\boxed{F_1^*}} \oplus \colorB{\boxed{F_1}}\\
V(\omega_{n}) &= \colorA{\boxed{\mb{C}}} \oplus \bigwedge^2 F_1 \oplus \cdots \oplus \colorB{\boxed{\bigwedge^{n-1} F_1}}
\end{align*}
Let $\mf{n}$ be the negative part of $\mf{g}$ in the $(n-1)$-grading, and let $R$ be its coordinate ring. By proceeding analogously to \S\ref{subsec:1nn1-n-even}, one obtains a resolution
\[
0 \to \colorB{\mb{C}} \otimes R \xto{\exp_{n}} \colorA{F_1} \otimes R \xto{\exp_1^*} \colorB{F_1^*}\otimes R \xto{\exp_{n-1}} \colorA{\CC}\otimes R
\]
Since the explicit details of the construction are nearly identical to what was described for the case of $n$ even, we omit them for brevity. The difference in output is described in Remark~\ref{rem:1nn1-even-odd-comparison}: one obtains the generic resolution of a Gorenstein ideal of codimension three on $n$ generators as per \cite{BE77}.

\subsubsection{Vertices $(n,n-3)$, format $(1,4,n,n-3)$, $n$ even}\label{subsec:14nx-n-even}
For this format, we draw the Dynkin diagram as
\[\begin{tikzcd}
\colorA{n} \ar[r,dash]& n-2 \ar[r,dash] & n-1\\
& \colorB{n-3} \ar[u,dash]\\
& \vdots \ar[u,dash]\\
& 1 \ar[u,dash]\\
\end{tikzcd}\]

We fix subalgebras $\mf{sl}(F_1)$, $\mf{sl}(F_2)$, and $\mf{sl}(F_3)$ of $\mf{g}$ as in \S\ref{sec:general-construction}, where $F_1 = \mb{C}^4$, $F_2 = \mb{C}^n$, and $F_3 = \mb{C}^{n-3}$.

The decompositions of $V(\omega_{1})$, $V(\omega_{n-1})$, and $V(\omega_n)$ in the gradings induced by the simple roots $\alpha_{n-3}$ and $\alpha_n$ are displayed below.

\begin{align*}
V(\omega_1) &= F_3 \oplus \bigwedge^2 F_1 \oplus \colorB{\boxed{F_3^* \otimes \bigwedge^4 F_1}}\\
&= \colorA{\boxed{F_2}} \oplus F_2^*\\
V(\omega_{n-1}) &= F_1^* \oplus F_3^* \otimes F_1 \oplus \bigwedge^2 F_3^* \otimes F_1^* \oplus \cdots \oplus \colorB{\boxed{\bigwedge^{n-3}F_3^* \otimes F_1}}\\
&= \colorA{\boxed{F_2^*}} \oplus \bigwedge^3 F_2^* \oplus \cdots \oplus \bigwedge^{n-1} F_2^*\\
V(\omega_n) &= F_1 \oplus F_3^* \otimes F_1^* \oplus \bigwedge^2 F_3^* \otimes F_1 \cdots \oplus \colorB{\boxed{\bigwedge^{n-3} F_3^* \otimes F_1^*}}\\
&= \colorA{\boxed{\mb{C}}} \oplus \bigwedge^2 F_2^* \oplus \cdots \oplus \bigwedge^n F_2^*
\end{align*}

We can either take $\mf{n}$ to be the negative part of $\mf{g}$ in the $n$-grading, or the negative part of $\mf{g}$ in the $(n-3)$-grading. Although the two resulting complexes appear quite different, they are equivalent in the sense of Lemma~\ref{lem:dual-format-equiv}.

Let us take the former option for now, so that
\[
	\mf{n} = \bigwedge^2 F_2
\]
with ($\mb{Z}^n$-graded) coordinate ring $R = \operatorname{Sym}_\mb{C}(\mf{n}^*)$. (From the perspective of the coarser $n$-grading, this is an ordinary polynomial ring with all generators in degree 1.) Let $f_1,\ldots,f_n$ be a standard basis of $F_2$. Then
\[
R = \mb{C}[\{x_{ij}\}_{1 \leq i < j \leq n}]
\]
with $x_{ij}$ dual to $f_i \wedge f_j$.

We construct a resolution of the form
\[
0 \to \colorB{F_3^*} \otimes R \xto{\exp_1} \colorA{F_2} \otimes R \xto{\exp_{n-1}^*} \colorB{F_1^*}\otimes R \xto{\exp_n} \colorA{\CC}\otimes R.
\]
The differential $d_3$ comes from the standard representation $V(\omega_1)$ so it is the simplest. The action of $\mf{n} = \bigwedge^2 F_2$ on $V(\omega_1) = F_2 \oplus F_2^*$ is given by the matrix
\[
	X = \begin{bmatrix}
	0_{n\times n} & B\\
	0_{n\times n} & 0_{n\times n}
	\end{bmatrix}, \qquad \exp X = \begin{bmatrix}
	I_n & B\\
	0_{n\times n} & I_n
	\end{bmatrix}
\]
where $B$ is the generic skew matrix on the variables $x_{ij}$. We want the part of $\exp X$ mapping from $F_3^* \otimes \bigwedge^4 F_1$ to $F_2$. Let $\phi_1,\ldots,\phi_n \in F_2^*$ be the dual basis of $f_1,\ldots,f_n$. The component $F_3^* \otimes \bigwedge^4 F_1 \subset V(\omega_1)$ has basis given by $\phi_4,\phi_5,\ldots,\phi_n$. Thus the differential $d_3$ is comprised of the last $n-3$ columns of $B$:
\[
	d_3 = \begin{bmatrix}
	-x_{14} & -x_{15} & \cdots & -x_{1n}\\
	-x_{24} & -x_{25} & \cdots & -x_{2n}\\
	-x_{34} & -x_{35} & \cdots & -x_{3n}\\
	0 & -x_{45} & \cdots & -x_{4n}\\
	x_{45} & 0 & \cdots & -x_{5n}\\
	\vdots & \vdots & \ddots & \vdots\\
	x_{4n} & x_{5n} & \cdots & 0
	\end{bmatrix}
\]
For the half-spinor representations, the action $X$ of $\mf{n}$ is described in \eqref{eq:halfspinoraction1} and \eqref{eq:halfspinoraction2} (replacing $e$ with $\phi$ and $b$ with $x$ throughout). The differential $d_2$ is (dual to) the part of $\exp X$ mapping from $\bigwedge^{n-3}F_3^* \otimes F_1 \subset V(\omega_{n-1})$ to $F_2^* \subset V(\omega_{n-1})$. For brevity, let $\phi_I = \bigwedge_{i\in I} \phi_i$ and $[n]=\{1,\ldots,n\}$. A basis for $\bigwedge^{n-3}F_3^* \otimes F_1 \subset V(\omega_{n-1})$ is given by
\[
	\phi_{[n]\setminus \{1\}}, \phi_{[n]\setminus \{2\}}, \phi_{[n]\setminus \{3\}} \in \bigwedge^{n-1}F_2^* \text{ and } \phi_{4,\ldots,n} \in \bigwedge^{n-3} F_2^*.
\]
Thus our matrix $d_2$ is comprised of the following Pfaffians of $B$:
\[
d_2 = \begin{bmatrix}
0 & P_{[n]\setminus \{1,2\}} & -P_{[n]\setminus \{1,3\}} & P_{[n]\setminus \{1,4\}} & -P_{[n]\setminus \{1,5\}} & \cdots & P_{[n]\setminus \{1,n\}} \\
-P_{[n]\setminus \{1,2\}} & 0 & P_{[n]\setminus \{2,3\}} & -P_{[n]\setminus \{2,4\}} & P_{[n]\setminus \{2,5\}} & \cdots & -P_{[n]\setminus \{2,n\}} \\
P_{[n]\setminus \{1,3\}} & -P_{[n]\setminus \{2,3\}} & 0 & P_{[n]\setminus \{3,4\}} & -P_{[n]\setminus \{3,5\}} & \cdots & P_{[n]\setminus \{3,n\}} \\
0 & 0 & 0 & P_{[n]\setminus \{1,2,3,4\}} & -P_{[n]\setminus \{1,2,3,5\}} & \cdots & P_{[n]\setminus \{1,2,3,n\}} 
\end{bmatrix}.
\]
Lastly, $d_1$ is the part of $\exp X$ mapping from $\bigwedge^{n-3}F_3^* \otimes F_1^* \subset V(\omega_n)$ to $\mb{C} \subset V(\omega_n)$. A basis for the former inside of $V(\omega_n)$ is given by
\[
	 \phi_{[n]\setminus \{2,3\}},\phi_{[n]\setminus \{1,3\}},\phi_{[n]\setminus \{1,2\}}\in \bigwedge^{n-2}F_2^* \text{ and } \phi_{1,\ldots,n} \in \bigwedge^n F_2^*.
\]
So we have
\[
	d_1 = \begin{bmatrix}
	-P_{[n]\setminus \{2,3\}} & -P_{[n]\setminus \{1,3\}} & -P_{[n]\setminus \{1,2\}} & P_{[n]}
	\end{bmatrix}.
\]
 
By our representation-theoretic construction, we have recovered the exact same resolution described in \cite[Proposition 3.2]{annebrown}. There it is proven that this resolution is generic for grade 3 almost complete intersections of odd type $n-3$.

One could repeat the preceding, taking $\mf{n}$ to be the negative part of $\mf{g}$ in the $(n-3)$-grading instead. In this case, the situation becomes slightly more complicated, because $\mf{n}$ is concentrated in two degrees:
\[
	\mf{n} = \underset{-2}{\Big(\bigwedge^2 F_3 \otimes \bigwedge^4 F_1^*\Big)} \oplus \underset{-1}{\Big(F_3 \otimes \bigwedge^2 F_1^*\Big)}.\\
\]
The resulting complex is then over the base ring $R=\mb{C}[\{b_{ij,k}\}_{1 \leq i < j \leq 4, 1 \leq k \leq n-3}, \{c_{hk}\}_{1 \leq h < k \leq n-3}]$, where $|b_{ij,k}|=1$ and $|c_{hk}|=2$ (in the coarser $(n-3)$-grading on $R$).

As the examples considered previously already illustrate how our construction for type $D_n$ recovers classical complexes (our primary goal for this section), we will not delve into this current example any further. For the curious, the resulting complex can be found in \cite[Theorem 5.1]{GW20}, although it was produced there by different means.

\subsubsection{Vertices $(n,n-3)$, format $(1,4,n,n-3)$, $n$ odd}\label{subsec:14nx-n-odd}
Lastly, we have the analogue of \S\ref{subsec:14nx-n-even} for $n$ odd. Since duality on the diagram $D_n$ interchanges the vertices $n-1$ and $n$ in this case, we will consider the bottom pieces in the $(n-1)$-grading instead of the $n$-grading as per Remark~\ref{rem:dual-node}.

Accordingly, to describe the graded components in the $(n-1)$-grading, we fix the subalgebra $\mf{sl}(F_2) \subset \mf{g}$ corresponding to the nodes $n, n-2, n-3, \ldots, 1$, where $F_2 = \mb{C}^n$. The subalgebras $\mf{sl}(F_1)$ and $\mf{sl}(F_3)$ follow the usual convention.

The decompositions of $V(\omega_{1})$, $V(\omega_{n-1})$, and $V(\omega_n)$ in the gradings induced by the simple roots $\alpha_{n-3}$ and $\alpha_{n-1}$ are displayed below.

\begin{align*}
V(\omega_1) &= F_3 \oplus \bigwedge^2 F_1 \oplus \colorB{\boxed{F_3^* \otimes \bigwedge^4 F_1}}\\
&= \colorA{\boxed{F_2}} \oplus F_2^*\\
V(\omega_{n-1}) &= F_1 \oplus F_3^* \otimes F_1^* \oplus \bigwedge^2 F_3^* \otimes F_1 \cdots \oplus \colorB{\boxed{\bigwedge^{n-3} F_3^* \otimes F_1}}\\
&= \colorA{\boxed{F_2^*}} \oplus \bigwedge^3 F_2^* \oplus \cdots \oplus \bigwedge^{n} F_2^*\\
V(\omega_n) &= F_1^* \oplus F_3^* \otimes F_1 \oplus \bigwedge^2 F_3^* \otimes F_1^* \oplus \cdots \oplus \colorB{\boxed{\bigwedge^{n-3}F_3^* \otimes F_1^*}}\\
&= \colorA{\boxed{\mb{C}}} \oplus \bigwedge^2 F_2^* \oplus \cdots \oplus \bigwedge^{n-1} F_2^*
\end{align*}

The rest is quite similar to the case of $n$ even, so we will be brief. Let $\mf{n} = \bigwedge^2 F_2$ be the negative part of $\mf{g}$ in the $(n-1)$-grading. Its ($\mb{Z}^n$-graded) coordinate ring is $R = \operatorname{Sym}_\mb{C}(\mf{n}^*)$. Taking explicit variables as we did for $n$ even, we have
\[
R = \mb{C}[\{x_{ij}\}_{1 \leq i < j \leq n}].
\]
Just as it was for the case of $n$ even, the differential $d_3$ is given by the last $n-3$ columns of the generic skew matrix $B$ on the variables $x_{ij}$.
\[
d_3 = \begin{bmatrix}
-x_{14} & -x_{15} & \cdots & -x_{1n}\\
-x_{24} & -x_{25} & \cdots & -x_{2n}\\
-x_{34} & -x_{35} & \cdots & -x_{3n}\\
0 & -x_{45} & \cdots & -x_{4n}\\
x_{45} & 0 & \cdots & -x_{5n}\\
\vdots & \vdots & \ddots & \vdots\\
x_{4n} & x_{5n} & \cdots & 0
\end{bmatrix}
\]
The differential $d_2$ is
\[
d_2 = \begin{bmatrix}
P_{[n]\setminus \{1\}} & -P_{[n]\setminus \{2\}} & P_{[n]\setminus \{3\}} & -P_{[n]\setminus \{4\}} & P_{[n]\setminus \{5\}} & \cdots & P_{[n]\setminus \{n\}}\\
0 & 0 & P_{[n]\setminus \{1,2,3\}} & -P_{[n]\setminus\{1,2,4\}} & P_{[n]\setminus\{1,2,5\}} & \cdots & P_{[n]\setminus\{1,2,n\}}\\
0 & P_{[n]\setminus \{1,2,3\}} & 0 & P_{[n]\setminus\{1,3,4\}} & -P_{[n]\setminus\{1,3,5\}} & \cdots & -P_{[n]\setminus\{1,3,n\}}\\
P_{[n]\setminus \{1,2,3\}} & 0 & 0 & -P_{[n]\setminus\{2,3,4\}} & P_{[n]\setminus\{2,3,5\}} & \cdots & P_{[n]\setminus\{2,3,n\}}
\end{bmatrix}
\]
and finally
\[
d_1 = \begin{bmatrix}
-P_{[n]\setminus \{1,2,3\}} & P_{[n] \setminus \{3\}} & P_{[n] \setminus \{2\}} & P_{[n] \setminus \{1\}}
\end{bmatrix}.
\]
Once again, this coincides with the complex described in \cite[Proposition 3.3]{annebrown}, corroborating Conjecture~\ref{conj:gen3} in this case.

\subsection{$E_6$}
In this subsection we consider the Lie algebra $\mf{g}$ associated to the Dynkin diagram $E_6$:
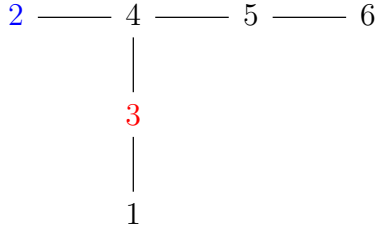
\[\begin{tikzcd}
\colorA{2} \ar[r,dash]& 4 \ar[r,dash] & 5 \ar[r,dash] & 6\\
& \colorB{3} \ar[u,dash]\\
& 1 \ar[u,dash]
\end{tikzcd}\]
We take $\mf{sl}(F_1), \mf{sl}(F_3) \subset \mf{g}$ as usual, but in view of Remark~\ref{rem:dual-node} we take $\mf{sl}(F_2)$ corresponding to the nodes $1,3,4,5,6$ as opposed to $6,5,4,3,1$. If we took the latter, then all instances of $F_2$ below should be replaced with $F_2^*$ and vice versa. 

The three fundamental representations of interest are $V(\omega_1)$, $V(\omega_6)$, and $V(\omega_2)$. Their decompositions in the $3$-grading and $2$-grading are displayed below.
\begin{align*}
	V(\omega_1) &= {F_1^*} \oplus {F_3^* \otimes F_1} \oplus {\bigwedge^2 F_3^* \otimes \bigwedge^3 F_1} \oplus {\colorB{\boxed{S_{2,1}F_3^* \otimes \bigwedge^5 F_1}}}\\
	&= {\colorA{\boxed{F_2}}} \oplus {\bigwedge^2 F_2^*} \oplus {\bigwedge^5 F_2^*}\\
	V(\omega_6) &= {F_3} \oplus {\bigwedge^2 F_1} \oplus {F_3^* \otimes \bigwedge^4 F_1} \oplus {\colorB{\boxed{\bigwedge^2 F_3^* \otimes S_{2,1^4} F_1}}}\\
	&= {\colorA{\boxed{F_2^*}}} \oplus {\bigwedge^4 F_2^*} \oplus {S_{2,1^5} F_2^*}\\
	V(\omega_2) &= {F_1} \oplus \cdots \oplus {\colorB{\boxed{S_{2,2}F_3^* \otimes S_{2^4,1} F_1}}}\\
	&= {\colorA{\boxed{\CC}}} \oplus \cdots \oplus{S_{2^6} F_2^*}
\end{align*}
The full decompositions of $V(\omega_2)$ can be found in \cite{LW19}.

To define $\mb{F}_i$ ($i = 2,3$), let $\mf{n}$ be the negative part of $\mf{g}$ in the $i$-grading. Let $R = \operatorname{Sym}_{\mb{C}}(\mf{n}^*)$ be the coordinate ring of $\mf{n}$ as usual. Both resolutions $\mb{F}_i$ have the form
\[
	0 \to \colorB{F_3^*} \otimes R \xto{\exp_{1}} \colorA{F_2} \otimes R \xto{\exp_6^*} \colorB{F_1^*}\otimes R \xto{\exp_2} \colorA{\CC}\otimes R
\]
where $\exp_i$ means the action of $\exp \mathfrak{n}$ on $V(\omega_i)$. The resolutions share the multigrading
\begin{align*}
0 \to
\bigoplus&\begin{matrix}
R(-(4, & \colorA{7}, & \colorB{9}, & 12, & 8, & 4))\\
R(-(5, & \colorA{7}, & \colorB{9}, & 12, & 8, & 4))
\end{matrix}
\to\\[1em]\xrightarrow{\exp_1}
\bigoplus&\begin{matrix}
R(-(3, & \colorA{5}, & \colorB{6}, & 8, & 5, & 2))\\
R(-(3, & \colorA{5}, & \colorB{6}, & 8, & 5, & 3))\\
R(-(3, & \colorA{5}, & \colorB{6}, & 8, & 6, & 3))\\
R(-(3, & \colorA{5}, & \colorB{6}, & 9, & 6, & 3))\\
R(-(3, & \colorA{5}, & \colorB{7}, & 9, & 6, & 3))\\
R(-(4, & \colorA{5}, & \colorB{7}, & 9, & 6, & 3))
\end{matrix}
\to\\[1em]\xrightarrow{\exp_6^*}
\bigoplus&\begin{matrix}
R(-(2, & \colorA{3}, & \colorB{4}, & 5, & 3, & 1))\\
R(-(2, & \colorA{3}, & \colorB{4}, & 5, & 3, & 2))\\
R(-(2, & \colorA{3}, & \colorB{4}, & 5, & 4, & 2))\\
R(-(2, & \colorA{3}, & \colorB{4}, & 6, & 4, & 2))\\
R(-(2, & \colorA{4}, & \colorB{4}, & 6, & 4, & 2))
\end{matrix}
\to\\\xrightarrow{\exp_2}
&R.
\end{align*}
We demonstrate how this grading can be inferred combinatorially, independently of computing the differentials themselves. Let $A$ denote the Cartan matrix of $E_6$. Its inverse
\[
	A^{-1} = \begin{bmatrix}
	4/3 & 1 & 5/3 & 2 & 4/3 & 2/3\\
	1 & 2 & 2 & 3 & 2 & 1\\
	5/3 & 2 & 10/3 & 4 & 8/3 & 4/3\\
	2 & 3 & 4 & 6 & 4 & 2\\
	4/3 & 2 & 8/3 & 4 & 10/3 & 5/3\\
	2/3 & 1 & 4/3 & 2 & 5/3 & 4/3
	\end{bmatrix}
\]
expresses the fundamental weights as linear combinations of the simple roots. So with respect to the simple roots, the weights in $F_1^* \subset V(\omega_2)$ have coordinates
\begin{align*}
	A^{-1}\omega_2 &= (1,2,2,3,2,1)\\
	A^{-1}\omega_2 - \alpha_2 &= (1,2,1,3,2,1)\\
	A^{-1}\omega_2 - \alpha_2 - \alpha_4 &= (1,2,1,2,2,1)\\
	A^{-1}\omega_2 - \alpha_2 - \alpha_4 - \alpha_5 &= (1,2,1,2,1,1)\\
	A^{-1}\omega_2 - \alpha_2 - \alpha_4 - \alpha_5 - \alpha_6 &= (1,2,1,2,1,0).
\end{align*}
The differential $d_1$ maps from $F_1^*$ to the lowest weight vector of $V(\omega_2)$, which has coordinates $A^{-1}(-\omega_2) = (-1,-2,-2,-3,-2,-1)$. Subtracting, we find the multigrading on $F_1^* \otimes R$ claimed above.

The weights in $F_2^* \subset V(\omega_6)$ have coordinates
\begin{align*}
	A^{-1}(-\omega_1) &= (-4/3,-1,-5/3,-2,-4/3,-2/3)\\
	A^{-1}(-\omega_1) + \alpha_1 &= (-1/3,-1,-5/3,-2,-4/3,-2/3)\\
	\vdots\\
	A^{-1}(-\omega_1) + \alpha_1 + \alpha_3 + \alpha_4 + \alpha_5 + \alpha_6 &= (-1/3,-1,-2/3,-1,-1/3,1/3).
\end{align*}
(Note that the lowest weight in $V(\omega_6)$ is $-\omega_1$.)
The highest weight in $F_1 \subset V(\omega_6)$ has coordinates $A^{-1}\omega_6 = (2/3, 1,4/3,2,5/3,4/3)$. When dualized, this should match with the lowest weight in $F_1^* \subset V(\omega_2)$. Thus the multigrading on $F_2 \otimes R$ can be obtained by subtracting the above coordinates for $F_2^* \subset V(\omega_6)$ from
\[
	(2/3,1,4/3,2,5/3,4/3) + (2,3,4,5,3,1) = (8/3,4,16/3,7,14/3,7/3).
\]
The multigrading on $F_3^* \otimes R$ can be inferred in a similar fashion by examining $V(\omega_1)$ so we omit the computation.
\begin{remark}
	Observe that by coarsening to the 2-grading (i.e. the second coordinate), we obtain the grading presented in \cite[Theorem 3.0.1]{SW21}. This connection to Schubert varieties will be explored in \S\ref{sec:proofres}.
\end{remark}

Next we describe the differentials explicitly. For $\mb{F}_2$, the nilpotent Lie algebra $\mf{n}$ is
\[
	\mf{n} = {\bigwedge^6 F_2} \oplus {\bigwedge^3 F_2}\\
\]
in degrees $-2$ and $-1$ respectively. Let $e_6,\ldots,e_1$ be a standard basis of $F_2$. Let $x_{ijk} \in \bigwedge^3 F_2^*$ be dual to $e_i \wedge e_j \wedge e_k$ and let $y \in \bigwedge^6 F_2^*$ be dual to $e_1\wedge \cdots \wedge e_6$. Then our base ring is
\[
R = \mb{C}[\{x_{ijk}\}_{1 \leq i < j < k \leq 6}, y]
\]
where $|x_{ijk}| = 1$ and $|y| = 2$.

The differential $d_3 = \exp_1$ is
\[
	\begin{bmatrix}
	- x_{236}x_{146} + x_{136}x_{246} - x_{126}x_{346} & X_2 \\
	X_1 & -x_{235}x_{145} + x_{135}x_{245} - x_{125}x_{345} \\
	-x_{234}x_{146} + x_{134}x_{246} - x_{124}x_{346} & -x_{234}x_{145} + x_{134}x_{245} - x_{124}x_{345}\\
	-x_{234}x_{136} + x_{134}x_{236} - x_{123}x_{346} & -x_{234}x_{135} + x_{134}x_{235} - x_{123}x_{345}\\
	-x_{234}x_{126} + x_{124}x_{236} - x_{123}x_{246} & -x_{234}x_{125} + x_{124}x_{235} - x_{123}x_{245}\\
	-x_{134}x_{126} + x_{124}x_{136} - x_{123}x_{146} & -x_{134}x_{125} + x_{124}x_{135} - x_{123}x_{145}
	\end{bmatrix}
\]
where
\begin{align*}
	X_1 &= \frac{1}{2}\Big(
	-x_{345}x_{126} + x_{245}x_{136} - x_{145}x_{236} - x_{235}x_{146} + x_{135}x_{246}\\
	&\quad -x_{125}x_{346} - x_{234}x_{156} + x_{134}x_{256} - x_{124}x_{356} + x_{123}x_{456}\Big) - y\\
	X_2 &= \frac{1}{2}\Big(
	-x_{345}x_{126} + x_{245}x_{136} - x_{145}x_{236} - x_{235}x_{146} + x_{135}x_{246}\\
	&\quad -x_{125}x_{346} + x_{234}x_{156} - x_{134}x_{256} + x_{124}x_{356} - x_{123}x_{456}\Big) + y.
\end{align*}
Before proceeding, we comment that although we printed this differential in its entirety, we will generally not do so for subsequent examples. For one, it would take too much space for the more complicated differentials. Moreover, the highly symmetric nature of the construction renders it unnecessary: after we describe some representative entries, the remaining ones can be obtained by appropriate permutations of the indices.

When we present each differential, we first give a ``table of entries'' tabulating the number of terms in each entry. Then, we describe only a few representative entries explicitly.

The table of entries for $d_2 = \exp_6^*$ is
\[
	\begin{bmatrix}
	3 & 3 & 3 & 3 & 3 & 11\\
	3 & 3 & 3 & 3 & 11 & 3\\
	3 & 3 & 3 & 11 & 3 & 3\\
	3 & 3 & 11 & 3 & 3 & 3\\
	0 & 0 & 1 & 1 & 1 & 1\\
	\end{bmatrix}.
\]
The 11-term entry in position $(1,6)$ is
\begin{gather*}
	\frac{1}{2}\Big(x_{345}x_{126} - x_{245}x_{136} - x_{145}x_{236} + x_{235}x_{146}\\
	+ x_{135}x_{246} - x_{125}x_{346} - x_{234}x_{156} - x_{134}x_{256}\\
	+ x_{124}x_{356} - x_{123}x_{456}\Big) + y.
\end{gather*}
The 3-term entry in position $(1,1)$ is
\begin{gather*}
	x_{134}x_{125} - x_{124}x_{135} + x_{123}x_{145}.
\end{gather*}
The last row is
\[
	\begin{bmatrix}
	0 & 0 & x_{123} & -x_{124} & x_{134} & -x_{234}
	\end{bmatrix}.
\]
Finally, the table of entries for $d_1 = \exp_2$ is
\[
	\begin{bmatrix}
	17 & 17 & 17 & 17 & 86
	\end{bmatrix}
\]
The 86-term entry has the form $\Delta + y^2$, where $\Delta$ is the unique $\mf{sl}(F_2)$-invariant (up to scale) of degree 4 on $\bigwedge^3 F_1$. The first 17-term entry is $-\frac{\partial \Delta}{\partial x_{156}} + x_{234}y$, and the others can be obtained by permuting the indices $\{1,2,3,4\}$. The appearance of this invariant and its partial derivatives will be discussed in \S\ref{sec:invariants}.

\subsection{$E_7$}
In this subsection we consider the Lie algebra $\mf{g}$ associated to the Dynkin diagram $E_7$. The vertex pair $(2,3)$ yields complexes of format $(1,6,7,2)$ while the vertex pair $(2,5)$ yields complexes of format $(1,5,7,3)$. In each case, we will only describe one differential of the complex $\mb{F}_2$ explicitly.
\subsubsection{Vertices $(2,3)$, format $(1,6,7,2)$}
For this format, we draw the Dynkin diagram as
\[\begin{tikzcd}
\colorA{2} \ar[r,dash]& 4 \ar[r,dash] & 5 \ar[r,dash] & 6 \ar[r,dash] & 7\\
& \colorB{3} \ar[u,dash]\\
& 1 \ar[u,dash]
\end{tikzcd}\]
Let $F_1 = \CC^6, F_2 = \CC^7$, and $F_3 = \CC^2$. Fix the following subalgebras of $\mf{g}$:
\begin{itemize}
	\item $\mf{sl}(F_3)$ corresponding to the node $1$,
	\item $\mf{sl}(F_1)$ corresponding to the nodes $7,6,5,4,2$,
	\item $\mf{sl}(F_2)$ corresponding to the nodes $7,6,5,4,3,1$.
\end{itemize}
The three fundamental representations of interest are $V(\omega_1)$, $V(\omega_7)$, and $V(\omega_2)$. Their full decompositions can be found in \cite{LW19}.

The representation $V(\omega_1)$ is the adjoint. It has 7 graded components in the 3-grading and 5 in the 2-grading:
\begin{align*}
V(\omega_1) &= F_3 \oplus \bigwedge^2 F_1 \oplus \cdots \oplus \colorB{\boxed{S_{3,2}F_3^* \otimes S_{2^6}F_1}}\\
&= \colorA{\boxed{F_2}} \oplus \bigwedge^4 F_2 \oplus \cdots \oplus S_{2^6,1} F_2
\end{align*}
The representation $V(\omega_7)$ has 5 graded components in the 3-grading and 4 in the 2-grading:
\begin{align*}
V(\omega_7) &= F_1^* \oplus F_3^* \otimes F_1 \oplus \cdots \oplus \colorB{\boxed{S_{2,2}F_3^* \otimes S_{2,1^5}F_1}}\\
&= \colorA{\boxed{F_2^*}} \oplus \bigwedge^2 F_2 \oplus \bigwedge^5 F_2 \oplus S_{2,1^6}F_2
\end{align*}
The representation $V(\omega_2)$ has 9 graded components in the 3-grading and 8 in the 2-grading:
\begin{align*}
V(\omega_2) &= F_1 \oplus F_3^* \otimes \bigwedge^3 F_1 \oplus \cdots \oplus \colorB{\boxed{S_{4,4}F_3^* \otimes S_{3^5,2} F_1}}\\
&= \colorA{\boxed{\mb{C}}} \oplus \bigwedge^3 F_2 \oplus \cdots \oplus S_{3^7} F_2
\end{align*}
Following the construction of \S\ref{sec:general-construction}, we obtain resolutions $\mb{F}_2$ and $\mb{F}_3$ of the form
\[
0 \to \colorB{F_3^*} \otimes R \xto{\exp_{1}} \colorA{F_2} \otimes R \xto{\exp_7^*} \colorB{F_1^*}\otimes R \xto{\exp_2} \colorA{\CC}\otimes R.
\]
Both resolutions have the multigrading
\begin{align*}
0 \to
\bigoplus&\begin{matrix}
R(-(8, & \colorA{13}, & \colorB{17}, & 24, & 18, & 12, & 6))\\
R(-(9, & \colorA{13}, & \colorB{17}, & 24, & 18, & 12, & 6))
\end{matrix}
\to\\[1em]\xrightarrow{\exp_1}
\bigoplus&\begin{matrix}
R(-(5, & \colorA{9}, & \colorB{11}, & 16, & 12, & 8, & 4))\\
R(-(6, & \colorA{9}, & \colorB{11}, & 16, & 12, & 8, & 4))\\
R(-(6, & \colorA{9}, & \colorB{12}, & 16, & 12, & 8, & 4))\\
R(-(6, & \colorA{9}, & \colorB{12}, & 17, & 12, & 8, & 4))\\
R(-(6, & \colorA{9}, & \colorB{12}, & 17, & 13, & 8, & 4))\\
R(-(6, & \colorA{9}, & \colorB{12}, & 17, & 13, & 9, & 4))\\
R(-(6, & \colorA{9}, & \colorB{12}, & 17, & 13, & 9, & 5))
\end{matrix}
\to\\[1em]\xrightarrow{\exp_7^*}
\bigoplus&\begin{matrix}
R(-(4, & \colorA{6}, & \colorB{8}, & 11, & 8, & 5, & 2))\\
R(-(4, & \colorA{6}, & \colorB{8}, & 11, & 8, & 5, & 3))\\
R(-(4, & \colorA{6}, & \colorB{8}, & 11, & 8, & 6, & 3))\\
R(-(4, & \colorA{6}, & \colorB{8}, & 11, & 9, & 6, & 3))\\
R(-(4, & \colorA{6}, & \colorB{8}, & 12, & 9, & 6, & 3))\\
R(-(4, & \colorA{7}, & \colorB{8}, & 12, & 9, & 6, & 3))
\end{matrix}
\to\\\xrightarrow{\exp_2}
&R.
\end{align*}
We will only describe the differential $d_2$ of $\mb{F}_2$ explicitly. This differential is the simplest because $V(\omega_7)$ has only 4 graded components in the 2-grading (thus the entries of the differential will have degree at most 3).

To that end, let $f_1,\ldots,f_7$ be a standard basis of $F_2$ and let
\[
	R = \mb{C}[\{x_{ijk}\}_{1 \leq i<j<k\leq 7}, \{y_{1\ldots \hat{i} \ldots 7}\}_{1 \leq i \leq 7}]
\]
where $x_{ijk}$ is corresponds to $f_i \wedge f_j \wedge f_k$ and similarly for $y_{1\ldots \hat{i} \ldots 7}$.

The table of entries for $d_2$ is
\[
	\begin{bmatrix}
	15 & 35 & 35 & 35 & 35 & 35 & 35\\
	35 & 15 & 35 & 35 & 35 & 35 & 35\\
	35 & 35 & 15 & 35 & 35 & 35 & 35\\
	35 & 35 & 35 & 15 & 35 & 35 & 35\\
	35 & 35 & 35 & 35 & 15 & 35 & 35\\
	3 & 3 & 3 & 3 & 3 & 11 & 11
	\end{bmatrix}.
\]
The 15-term entry in position $(1,1)$ is
\begin{equation}\label{eq:15-term}
	\begin{split}
	x_{145}x_{136}x_{127} - x_{135}x_{146}x_{127} + x_{134}x_{156}x_{127} - x_{145}x_{126}x_{137} + x_{125}x_{146}x_{137}\\
	- x_{124}x_{156}x_{137} + x_{135}x_{126}x_{147} - x_{125}x_{136}x_{147} + x_{123}x_{156}x_{147} - x_{134}x_{126}x_{157}\\
	+ x_{124}x_{136}x_{157} - x_{123}x_{146}x_{157} + x_{134}x_{125}x_{167} - x_{124}x_{135}x_{167}+x_{123}x_{145}x_{167}.
	\end{split}
\end{equation}
The 35-term entry in position $(1,2)$ is
\begin{equation}\label{eq:35-term}
	\begin{split}
	\frac{1}{2}\Big(
	x_{245}x_{136}x_{127} + x_{145}x_{236}x_{127} - x_{235}x_{146}x_{127} - x_{135}x_{246}x_{127} + x_{234}x_{156}x_{127}\\
	+ x_{134}x_{256}x_{127} - x_{245}x_{126}x_{137} + x_{125}x_{246}x_{137} - x_{124}x_{256}x_{137} - x_{145}x_{126}x_{237}\\
	+ x_{125}x_{146}x_{237} - x_{124}x_{156}x_{237} + x_{235}x_{126}x_{147} - x_{125}x_{236}x_{147} + x_{123}x_{256}x_{147}\\
	+ x_{135}x_{126}x_{247} - x_{125}x_{136}x_{247} + x_{123}x_{156}x_{247} - x_{234}x_{126}x_{157} + x_{124}x_{236}x_{157}\\
	- x_{123}x_{246}x_{157}- x_{134}x_{126}x_{257} + x_{124}x_{136}x_{257} - x_{123}x_{146}x_{257} + x_{234}x_{125}x_{167}\\
	- x_{124}x_{235}x_{167} + x_{123}x_{245}x_{167} + x_{134}x_{125}x_{267} - x_{124}x_{135}x_{267} + x_{123}x_{145}x_{267}
	\Big)\\
	- x_{127}y_{123456} + x_{126}y_{123457} - x_{125}y_{123467} + x_{124}y_{123567} - x_{123}y_{124567}.
	\end{split}
\end{equation}
The 3-term entry in position $(6,1)$ is
\[
	x_{134}x_{125} - x_{124}x_{135} + x_{123}x_{145}.
\]
The 11-term entry in position $(6,6)$ is
\begin{gather*}
	\frac{1}{2}\Big(
	x_{345}x_{126} - x_{245}x_{136} + x_{145}x_{236} - x_{135}x_{246} + x_{125}x_{346} - x_{234}x_{156}\\
	+ x_{134}x_{256} - x_{124}x_{356} + x_{123}x_{456}
	\Big) - y_{123456}.
\end{gather*}

\subsubsection{Vertices $(2,5)$, format $(1,5,7,3)$}
For this format, we draw the Dynkin diagram as
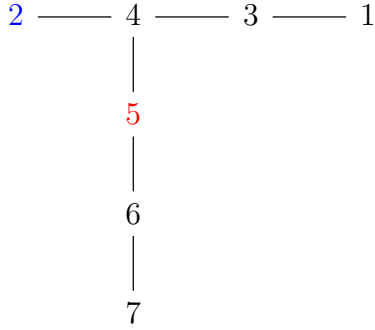
\[\begin{tikzcd}
\colorA{2} \ar[r,dash]& 4 \ar[r,dash] & 3 \ar[r,dash] & 1 \\
& \colorB{5} \ar[u,dash]\\
& 6 \ar[u,dash]\\
& 7 \ar[u,dash]
\end{tikzcd}\]
Let $F_1 = \CC^5, F_2 = \CC^7$, and $F_3 = \CC^3$. Fix the following subalgebras of $\mf{g}$:
\begin{itemize}
	\item $\mf{sl}(F_3)$ corresponding to the nodes $6,7$,
	\item $\mf{sl}(F_1)$ corresponding to the nodes $1,3,4,2$,
	\item $\mf{sl}(F_2)$ corresponding to the nodes $1,3,4,5,6,7$.
\end{itemize}
The three fundamental representations of interest are $V(\omega_7)$, $V(\omega_1)$, and $V(\omega_2)$. Their full decompositions can be found in \cite{LW19}.

The representation $V(\omega_7)$ has 6 graded components in the 5-grading and 4 in the 2-grading:
\begin{align*}
V(\omega_7) &= F_3 \oplus \bigwedge^2 F_1 \oplus \cdots \oplus \colorB{\boxed{S_{2,1,1}F_3^* \otimes S_{2^5}F_1}}\\
&= \colorA{\boxed{F_2}} \oplus \bigwedge^2 F_2^* \oplus \bigwedge^5 F_2^* \oplus S_{2,1^6}F_2^*
\end{align*}
The representation $V(\omega_1)$ is the adjoint. It has 7 graded components in the 5-grading and 5 in the 2-grading:
\begin{align*}
V(\omega_1) &= F_1^* \oplus F_3^* \otimes F_1 \oplus \cdots \oplus \colorB{\boxed{S_{2,2,2}F_3^* \otimes S_{3,2^4}F_1}}\\
&= \colorA{\boxed{F_2^*}} \oplus \bigwedge^4 F_2^* \oplus \cdots \oplus S_{2^6,1} F_2^*
\end{align*}
The representation $V(\omega_2)$ has 10 graded components in the 5-grading and 8 in the 2-grading:
\begin{align*}
V(\omega_2) &= F_1 \oplus F_3^* \otimes \bigwedge^3 F_1 \oplus \cdots \oplus \colorB{\boxed{S_{3,3,3}F_3^* \otimes S_{4^4,3} F_1}}\\
&= \colorA{\boxed{\mb{C}}} \oplus \bigwedge^3 F_2^* \oplus \cdots \oplus S_{3^7} F_2^*
\end{align*}
Following the construction of \S\ref{sec:general-construction}, we obtain resolutions $\mb{F}_2$ and $\mb{F}_5$ of the form
\[
0 \to \colorB{F_3^*} \otimes R \xto{\exp_{7}} \colorA{F_2} \otimes R \xto{\exp_1^*} \colorB{F_1^*}\otimes R \xto{\exp_2} \colorA{\CC}\otimes R.
\]
Both resolutions have the multigrading
\begin{align*}
0 \to
\bigoplus&\begin{matrix}
R(-(8, & \colorA{13}, & 16, & 24, & \colorB{19}, & 12, & 6))\\
R(-(8, & \colorA{13}, & 16, & 24, & \colorB{19}, & 13, & 6))\\
R(-(8, & \colorA{13}, & 16, & 24, & \colorB{19}, & 13, & 7))
\end{matrix}
\to\\[1em]\xrightarrow{\exp_7}
\bigoplus&\begin{matrix}
R(-(6, & \colorA{10}, & 12, & 18, & \colorB{14}, & 9, & 4))\\
R(-(6, & \colorA{10}, & 12, & 18, & \colorB{14}, & 9, & 5))\\
R(-(6, & \colorA{10}, & 12, & 18, & \colorB{14}, & 10, & 5))\\
R(-(6, & \colorA{10}, & 12, & 18, & \colorB{15}, & 10, & 5))\\
R(-(6, & \colorA{10}, & 12, & 19, & \colorB{15}, & 10, & 5))\\
R(-(6, & \colorA{10}, & 13, & 19, & \colorB{15}, & 10, & 5))\\
R(-(7, & \colorA{10}, & 13, & 19, & \colorB{15}, & 10, & 5))
\end{matrix}
\to\\[1em]\xrightarrow{\exp_1^*}
\bigoplus&\begin{matrix}
R(-(3, & \colorA{6}, & 7, & 11, & \colorB{9}, & 6, & 3))\\
R(-(4, & \colorA{6}, & 7, & 11, & \colorB{9}, & 6, & 3))\\
R(-(4, & \colorA{6}, & 8, & 11, & \colorB{9}, & 6, & 3))\\
R(-(4, & \colorA{6}, & 8, & 12, & \colorB{9}, & 6, & 3))\\
R(-(4, & \colorA{7}, & 8, & 12, & \colorB{9}, & 6, & 3))
\end{matrix}
\to\\\xrightarrow{\exp_2}
&R
\end{align*}
If we take $f_7,\ldots,f_1$ to be a standard basis of $F_2$ and define $R$ as in the previous subsection, then the differential $d_3$ of $\mb{F}_2$ is the transpose of the first three rows of $d_2$ for $(1,6,7,2)$.

\subsection{$E_8$}
Finally, we consider the Lie algebra $\mf{g}$ associated to the Dynkin diagram $E_8$. The vertex pair $(2,3)$ yields complexes of format $(1,7,8,2)$ and the pair $(2,5)$ yields complexes of format $(1,5,8,4)$. The construction is completely analogous to that for $E_7$, as there is once again no exceptional duality. Unfortunately, all of the differentials are too complicated to describe here explicitly. Only the simplest differential in each case has been implemented in Macaulay2, namely $d_2$ for $(1,7,8,2)$ and $d_3$ for $(1,5,8,4)$. These differentials are constructed from the adjoint representation $V(\omega_8)$, which is the smallest of the three extremal fundamental representations.

\subsubsection{Vertices $(2,3)$, format $(1,7,8,2)$}\label{sec:1782}
For this format, we draw the Dynkin diagram as
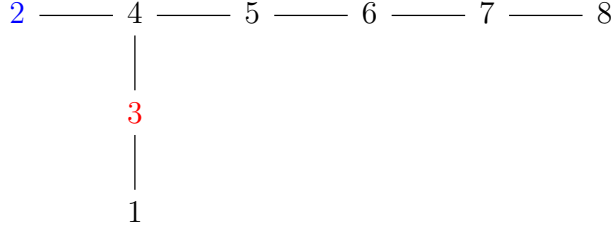
\[\begin{tikzcd}
\colorA{2} \ar[r,dash]& 4 \ar[r,dash] & 5 \ar[r,dash] & 6 \ar[r,dash] & 7 \ar[r,dash] & 8\\
& \colorB{3} \ar[u,dash]\\
& 1 \ar[u,dash]
\end{tikzcd}\]
Let $F_1 = \CC^7, F_2 = \CC^8$, and $F_3 = \CC^2$. Fix the following subalgebras of $\mf{g}$:
\begin{itemize}
	\item $\mf{sl}(F_3)$ corresponding to the node $1$,
	\item $\mf{sl}(F_1)$ corresponding to the nodes $8,7,6,5,4,2$,
	\item $\mf{sl}(F_2)$ corresponding to the nodes $8,7,6,5,4,3,1$.
\end{itemize}
The three fundamental representations of interest are $V(\omega_1)$, $V(\omega_8)$, and $V(\omega_2)$. Their full decompositions can be found in \cite{LW19}.

The representation $V(\omega_1)$ has 15 graded components in the 3-grading and 11 in the 2-grading:
\begin{align*}
V(\omega_1) &= F_3 \oplus \bigwedge^2 F_1 \oplus \cdots \oplus \colorB{\boxed{S_{7,6}F_3^* \otimes S_{4^7}F_1}}\\
&= \colorA{\boxed{F_2}} \oplus \bigwedge^4 F_2 \oplus \cdots \oplus S_{4^7,3} F_2
\end{align*}
The representation $V(\omega_8)$ is the adjoint. It has 9 graded components in the 3-grading and 7 in the 2-grading:
\begin{align*}
V(\omega_8) &= F_1^* \oplus F_3^* \otimes F_1 \oplus \cdots \oplus \colorB{\boxed{S_{4,4}F_3^* \otimes S_{3,2^6}F_1}}\\
&= \colorA{\boxed{F_2^*}} \oplus \bigwedge^2 F_2 \oplus \cdots \oplus S_{3,2^7}F_2
\end{align*}
The representation $V(\omega_2)$ has 21 graded components in the 3-grading and 17 in the 2-grading:
\begin{align*}
V(\omega_2) &= F_1 \oplus F_3^* \otimes \bigwedge^3 F_1 \oplus \cdots \oplus \colorB{\boxed{S_{10,10}F_3^* \otimes S_{6^6,5} F_1}}\\
&= \colorA{\boxed{\mb{C}}} \oplus \bigwedge^3 F_2 \oplus \cdots \oplus S_{6^8} F_2
\end{align*}
Following the construction of \S\ref{sec:general-construction}, we obtain resolutions $\mb{F}_2$ and $\mb{F}_3$ of the form
\[
0 \to \colorB{F_3^*} \otimes R \xto{\exp_{1}} \colorA{F_2} \otimes R \xto{\exp_8^*} \colorB{F_1^*}\otimes R \xto{\exp_2} \colorA{\CC}\otimes R.
\]
Both resolutions have the multigrading
\begin{align*}
0 \to
\bigoplus&\begin{matrix}
R(-(20, & \colorA{31}, & \colorB{41}, & 60, & 48, & 36, & 24, & 12))\\
R(-(21, & \colorA{31}, & \colorB{41}, & 60, & 48, & 36, & 24, & 12))
\end{matrix}
\to\\[1em]\xrightarrow{\exp_1}
\bigoplus&\begin{matrix}
R(-(13, & \colorA{21}, & \colorB{27}, & 40, & 32, & 24, & 16, & 8))\\
R(-(14, & \colorA{21}, & \colorB{27}, & 40, & 32, & 24, & 16, & 8))\\
R(-(14, & \colorA{21}, & \colorB{28}, & 40, & 32, & 24, & 16, & 8))\\
R(-(14, & \colorA{21}, & \colorB{28}, & 41, & 32, & 24, & 16, & 8))\\
R(-(14, & \colorA{21}, & \colorB{28}, & 41, & 33, & 24, & 16, & 8))\\
R(-(14, & \colorA{21}, & \colorB{28}, & 41, & 33, & 25, & 16, & 8))\\
R(-(14, & \colorA{21}, & \colorB{28}, & 41, & 33, & 25, & 17, & 8))\\
R(-(14, & \colorA{21}, & \colorB{28}, & 41, & 33, & 25, & 17, & 9))
\end{matrix}
\to\\[1em]\xrightarrow{\exp_8^*}
\bigoplus&\begin{matrix}
R(-(10, & \colorA{15}, & \colorB{20}, & 29, & 23, & 17, & 11, & 5))\\
R(-(10, & \colorA{15}, & \colorB{20}, & 29, & 23, & 17, & 11, & 6))\\
R(-(10, & \colorA{15}, & \colorB{20}, & 29, & 23, & 17, & 12, & 6))\\
R(-(10, & \colorA{15}, & \colorB{20}, & 29, & 23, & 18, & 12, & 6))\\
R(-(10, & \colorA{15}, & \colorB{20}, & 29, & 24, & 18, & 12, & 6))\\
R(-(10, & \colorA{15}, & \colorB{20}, & 30, & 24, & 18, & 12, & 6))\\
R(-(10, & \colorA{16}, & \colorB{20}, & 30, & 24, & 18, & 12, & 6))
\end{matrix}
\to\\\xrightarrow{\exp_2}
&R.
\end{align*}

\subsubsection{Vertices $(2,5)$, format $(1,5,8,4)$}
For this format, we draw the Dynkin diagram as
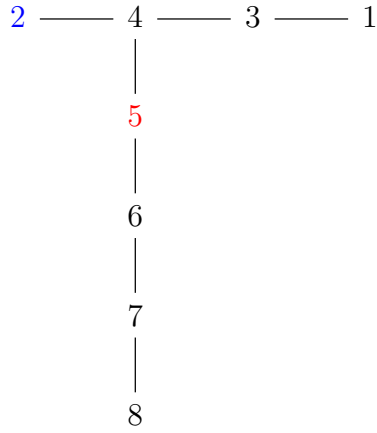
\[\begin{tikzcd}
\colorA{2} \ar[r,dash]& 4 \ar[r,dash] & 3 \ar[r,dash] & 1 \\
& \colorB{5} \ar[u,dash]\\
& 6 \ar[u,dash]\\
& 7 \ar[u,dash]\\
& 8 \ar[u,dash]
\end{tikzcd}\]
Let $F_1 = \CC^5, F_2 = \CC^8$, and $F_3 = \CC^4$. Fix the following subalgebras of $\mf{g}$:
\begin{itemize}
	\item $\mf{sl}(F_3)$ corresponding to the nodes $6,7,8$,
	\item $\mf{sl}(F_1)$ corresponding to the nodes $1,3,4,2$,
	\item $\mf{sl}(F_2)$ corresponding to the nodes $1,3,4,5,6,7,8$.
\end{itemize}
The three fundamental representations of interest are $V(\omega_8)$, $V(\omega_1)$, and $V(\omega_2)$. Their full decompositions can be found in \cite{LW19}.

The representation $V(\omega_8)$ is the adjoint. It has 11 graded components in the 5-grading and 7 in the 2-grading:
\begin{align*}
V(\omega_8) &= F_3 \oplus \bigwedge^2 F_1 \oplus \cdots \oplus \colorB{\boxed{S_{3,2^3}F_3^* \otimes S_{4^5}F_1}}\\
&= \colorA{\boxed{F_2}} \oplus \bigwedge^2 F_2^* \oplus \cdots \oplus S_{3,2^7}F_2^*
\end{align*}
The representation $V(\omega_1)$ has 17 graded components in the 5-grading and 11 in the 2-grading:
\begin{align*}
V(\omega_1) &= F_1^* \oplus F_3^* \otimes F_1 \oplus \cdots \oplus \colorB{\boxed{S_{4^4}F_3^* \otimes S_{7,6^4}F_1}}\\
&= \colorA{\boxed{F_2^*}} \oplus \bigwedge^4 F_2^* \oplus \cdots \oplus S_{4^7,3} F_2^*
\end{align*}
The representation $V(\omega_2)$ has 25 graded components in the 5-grading and 17 in the 2-grading:
\begin{align*}
V(\omega_2) &= F_1 \oplus F_3^* \otimes \bigwedge^3 F_1 \oplus \cdots \oplus \colorB{\boxed{S_{6^4}F_3^* \otimes S_{10^4,9} F_1}}\\
&= \colorA{\boxed{\mb{C}}} \oplus \bigwedge^3 F_2^* \oplus \cdots \oplus S_{6^8} F_2^*
\end{align*}
Following the construction of \S\ref{sec:general-construction}, we obtain resolutions $\mb{F}_2$ and $\mb{F}_5$ of the form
\[
0 \to \colorB{F_3^*} \otimes R \xto{\exp_{8}} \colorA{F_2} \otimes R \xto{\exp_1^*} \colorB{F_1^*}\otimes R \xto{\exp_2} \colorA{\CC}\otimes R.
\]
Both resolutions have the multigrading
\begin{align*}
0 \to
\bigoplus&\begin{matrix}
R(-(20, & \colorA{31}, & 40, & 60, & \colorB{49}, & 36, & 24, & 12))\\
R(-(20, & \colorA{31}, & 40, & 60, & \colorB{49}, & 37, & 24, & 12))\\
R(-(20, & \colorA{31}, & 40, & 60, & \colorB{49}, & 37, & 25, & 12))\\
R(-(20, & \colorA{31}, & 40, & 60, & \colorB{49}, & 37, & 25, & 13))
\end{matrix}
\to\\[1em]\xrightarrow{\exp_8}
\bigoplus&\begin{matrix}
R(-(16, & \colorA{25}, & 32, & 48, & \colorB{39}, & 29, & 19, & 9))\\
R(-(16, & \colorA{25}, & 32, & 48, & \colorB{39}, & 29, & 19, & 10))\\
R(-(16, & \colorA{25}, & 32, & 48, & \colorB{39}, & 29, & 20, & 10))\\
R(-(16, & \colorA{25}, & 32, & 48, & \colorB{39}, & 30, & 20, & 10))\\
R(-(16, & \colorA{25}, & 32, & 48, & \colorB{40}, & 30, & 20, & 10))\\
R(-(16, & \colorA{25}, & 32, & 49, & \colorB{40}, & 30, & 20, & 10))\\
R(-(16, & \colorA{25}, & 33, & 49, & \colorB{40}, & 30, & 20, & 10))\\
R(-(17, & \colorA{25}, & 33, & 49, & \colorB{40}, & 30, & 20, & 10))
\end{matrix}
\to\\[1em]\xrightarrow{\exp_1^*}
\bigoplus&\begin{matrix}
R(-(9, & \colorA{15}, & 19, & 29, & \colorB{24}, & 18, & 12, & 6))\\
R(-(10, & \colorA{15}, & 19, & 29, & \colorB{24}, & 18, & 12, & 6))\\
R(-(10, & \colorA{15}, & 20, & 29, & \colorB{24}, & 18, & 12, & 6))\\
R(-(10, & \colorA{15}, & 20, & 30, & \colorB{24}, & 18, & 12, & 6))\\
R(-(10, & \colorA{16}, & 20, & 30, & \colorB{24}, & 18, & 12, & 6))
\end{matrix}
\to\\\xrightarrow{\exp_2}
&R.
\end{align*}


\section{Gorenstein ideals of codimension four}\label{sec:Gor4}
We now explicitly describe the second family of examples, which yield resolutions of Gorenstein ideals of codimension four on $n = 4,\ldots,8$ generators. We conjecture that these resolutions are generic (for their respective number of generators).

Here we consider Dynkin diagrams of the form
\[\begin{tikzcd}[column sep = small, row sep = small]
	\colorC{x_1} \ar[r,dash] & u \ar[r,dash]\ar[d,dash] & y_1 \ar[r,dash] & \cdots \ar[r,dash] & y_{n-4} \\
& z_1 \ar[d,dash]\\
& \colorB{z_2}	
\end{tikzcd}\]
As $n$ ranges from 4 to 8, we obtain the diagrams $E_4 = A_4, E_5 = D_5, E_6, E_7, E_8$ respectively. For each case, we obtain two equivalent complexes $\mb{F}_{x_1}$ and $\mb{F}_{z_2}$, as per Lemma~\ref{lem:dual-format-equiv}. We will just describe the complex $\mb{F}_{x_1}$.

\subsection{$E_4 = A_4$}
For $n=4$, we have the diagram $A_4$, displayed below with the standard Bourbaki numbering, and its corresponding Lie algebra $\mf{g} = \mf{sl}(5)$.
\[\begin{tikzcd}[column sep = small, row sep = small]
\colorC{4} \ar[r,dash] & 3 \ar[d,dash]\\
& 2 \ar[d,dash]\\
& \colorB{1}
\end{tikzcd}\]
The representations of interest are $V(\omega_1)$ and $V(\omega_3)$. Per Remark~\ref{rem:dual-node}, since the nodes 4 and 1 are dual on the Dynkin diagram, we will use the bottom graded components in the 1-grading rather than the 4-grading. To faciliate this, let $F = \mb{C}^4$, and let write $\mf{sl}(F)$ for the subalgebra of $\mf{g}$ corresponding to the nodes $4,3,2$. The decompositions are

\begin{align*}
V(\omega_{1}) &= \colorC{\boxed{F^*}} \oplus \colorB{\boxed{\mb{C}}}\\
V(\omega_3) &= \colorC{\boxed{F}} \oplus \colorB{\boxed{\bigwedge^2 F}}
\end{align*}

Let $\mf{n}$ be the negative part of $\mf{g}$ in the 1-grading. Explicitly, $\mf{n} = F^*$. As before, let $R = \operatorname{Sym}_{\mb{C}}(\mf{n}^*)$, i.e.
\[
	R = \mb{C}[x_1,x_2,x_3,x_4]
\]
where $x_1,\ldots,x_4$ is a basis of $F$. Our construction yields a resolution
\[
0 \to \colorB{\mb{C}} \otimes R \xto{\exp_1} \colorC{F^*} \otimes R \xto{\exp_3^*} \colorB{\bigwedge^2 F}\otimes R \xto{\exp_3} \colorC{F}\otimes R \xto{\exp_1^*} \colorB{\mb{C}}\otimes R.
\]
Note that $\bigwedge^2 F \cong \bigwedge^2 F^*$. The actions of $\mf{n}$ on $V(\omega_1)$ and $V(\omega_3)$ are the evident ones, and from there it is straightforward to show that this resolution is none other than the Koszul complex on the variables $x_1,\ldots,x_4$. So the Genericity Conjecture holds in this very simple case, as this is clearly the generic resolution of a Gorenstein ideal of codimension four on $n=4$ generators (i.e. a complete intersection).

\subsection{$E_5 = D_5$}
For $n=5$, we have the diagram $D_5$ and its corresponding Lie algebra $\mf{g} = \mf{so}(10)$.
\[\begin{tikzcd}[column sep = small, row sep = small]
\colorC{5} \ar[r,dash] & 3 \ar[d,dash]\ar[r,dash] & 4\\
& 2 \ar[d,dash]\\
& \colorB{1}
\end{tikzcd}\]
Since the nodes 5 and 4 are dual on $D_5$, we will use the bottom graded components in the 4-grading rather than the 5-grading. We use the top graded components in the 1-grading as usual. Let $\mf{sl}(F)$ be the subalgebra of $\mf{g}$ corresponding to the nodes $1,2,3,5$ (where $F = \mb{C}^5$) and let $\mf{sl}(H)$ be the subalgebra of $\mf{g}$ corresponding to the nodes $2,3,5$ (where $H = \mb{C}^4$). The decompositions of $V(\omega_1)$ in the 1-grading and 4-grading are
\begin{align*}
V(\omega_{1}) &= \mb{C} \oplus (H\oplus H^*) \oplus \colorB{\boxed{\mb{C}}}\\
&= \colorC{\boxed{F^*}} \oplus F
\end{align*}
Note that in the 1-grading, the components are representations of $\mf{so}(H\oplus H^*)$. The other representation of interest is $V(\omega_4)$:
\begin{align*}
V(\omega_4) &= (H \oplus \bigwedge^3 H) \oplus \colorB{\boxed{(\mb{C} \oplus \bigwedge^2 H \oplus \bigwedge^4 H)}}\\
&= \colorC{\boxed{F}} \oplus \bigwedge^3 F \oplus \bigwedge^5 F	
\end{align*}
Let $\mf{n}$ be the negative part of $\mf{g}$ in the 4-grading, i.e. $\mf{n} = \bigwedge^2 F^*$, and let $R$ be its coordinate ring as usual.

Our construction yields a resolution
\[
	0 \to \colorB{\mb{C}} \otimes R \xto{\exp_1} \colorC{F^*} \otimes R \xto{\exp_4^*} \colorB{(\mb{C} \oplus \bigwedge^2 H \oplus \bigwedge^4 H)}\otimes R \xto{\exp_4} \colorC{F}\otimes R \xto{\exp_1^*} \colorB{\mb{C}}\otimes R.
\]
However, this resolution is not minimal. In the representation $V(\omega_4)$, there is a 1-dimensional overlap between the two boxed components since $F = H \oplus \mb{C}$, causing units to appear in $\exp_4$. After simplification, we obtain the isomorphic minimal complex
\[
	0 \to \mb{C} \otimes R \xto{\exp_1} H^* \otimes R \xto{\exp_4^*} \bigwedge^2 H \otimes R \xto{\exp_4} H \otimes R \xto{\exp_1^*} \mb{C} \otimes R
\]
but this is just a Koszul complex on $H \subset \bigwedge^2 F$ (recall that $R = \operatorname{Sym}_{\mb{C}}(\bigwedge^2 F)$).

Once again, the Genericity Conjecture holds: it is known that a Gorenstein ideal of codimension $c$ cannot be minimally generated by $c+1$ elements \cite{Kunz79}, thus in codimension four, any $(1,5,8,5,1)$ complex resolving a Gorenstein ideal must be Koszul plus split.

\subsection{$E_6$}
The previous two examples merely resulted in Koszul complexes, but for $n=6$ the construction yields a more interesting output. Now we deal with $E_6$:
\[\begin{tikzcd}[column sep = small, row sep = small]
\colorC{2} \ar[r,dash] & 4 \ar[d,dash]\ar[r,dash] & 5 \ar[r,dash] & 6\\
& 3 \ar[d,dash]\\
& \colorB{1}
\end{tikzcd}\]
Write $\mf{g}$ for the corresponding Lie algebra. We fix some subalgebras of $\mf{g}$ as follows:
\begin{itemize}
	\item $\mf{sl}(H)$ corresponding to the nodes $6,5,4,3$, where $H = \mb{C}^5$,
	\item $\mf{sl}(F)$ corresponding to the nodes $6,5,4,3,1$, where $F = \mb{C}^6$.
\end{itemize}
Note that the subalgebra corresponding to the nodes $6,5,4,3,2$ is $\mf{so}(H\oplus H^*)$.

With this notation, the representations $V(\omega_1)$ and $V(\omega_6)$ decompose in the 1-grading and 2-grading as
\begin{align*}
V(\omega_1) &= (H \oplus H^*) \oplus {V(\omega_3,\mf{so}(H \oplus H^*))} \oplus {\colorB{\boxed{\CC}}}\\
&= {\colorA{\boxed{F^*}}} \oplus {\bigwedge^2 F} \oplus {\bigwedge^5 F}\\
V(\omega_6) &= {\CC} \oplus {V(\omega_2,\mf{so}(H \oplus H^*))} \oplus {\colorB{\boxed{(H \oplus H^*)}}}\\
&= {\colorA{\boxed{F}}} \oplus {\bigwedge^4 F} \oplus {S_{2,1^5} F}
\end{align*}
Let $\mf{n}$ be the negative part of $\mf{g}$ in the 2-grading:
\[
	\mf{n} = \underset{-2}{\bigwedge^6 F^*} \oplus \underset{-1}{\bigwedge^3 F^*}.
\]
Let $e_1,\ldots,e_6$ be a standard basis of $F$. Let $\epsilon_1,\ldots,\epsilon_6$ be the dual basis. As usual, we set $R = \operatorname{Sym}_{\mb{C}}(\mf{n}^*)$ to be the coordinate ring of $\mf{n}$, i.e.
\[
	R = \mb{C}[\{x_{ijk}\}_{1 \leq i < j < k \leq 6}, y]
\]
where $x_{ijk}$ corresponds to $e_{ijk}$ (shorthand for $e_i \wedge e_j \wedge e_k$) and $y$ corresponds to $e_{1\ldots 6}$.

We produce a free resolution of the form
\[
0 \to \colorB{\CC} \otimes R \xto{\exp_1} \colorA{F^*} \otimes R \xto{\exp_6^*} \colorB{(H \oplus H^*)}\otimes R \xto{\exp_6} \colorA{F}\otimes R \xto{\exp_1^*} \colorB{\CC} \otimes R.
\]
Since it is self-dual, it will suffice to describe $d_4 = \exp_1$ and $d_3 = \exp_6^*$.

A generic element of $\mf{n}$ has the form
\[
	\sum x_{ijk} \epsilon_{ijk} + y \epsilon_{1\ldots 6}
\]
where $\epsilon_I$ is shorthand for $\bigwedge_{i\in I} \epsilon_i$. A highest weight vector in $V(\omega_1)$ is $e_{1\ldots 5}$, so to determine $d_4$ we need only exponentiate the action of the above element on $e_{1 \ldots 5}$, i.e.
\[
	\left( \frac{1}{2}\left(\sum x_{ijk} \epsilon_{ijk} \right)^2 + y\epsilon_{1\ldots 6}\right) \cdot e_{1 \ldots 5} \in F^* \otimes R.
\]
Computation shows that the coefficients of $\epsilon_1,\ldots,\epsilon_5$ in the expression above are the $4\times 4$ Pfaffians of the $5\times 5$ generic skew matrix
\[
	\begin{bmatrix}
	0 & x_{345} & -x_{245} & x_{235} & -x_{234}\\
	-x_{345} & 0 & x_{145} & -x_{135} & x_{134}\\
	x_{245} & -x_{145} & 0 & x_{125} & -x_{124}\\
	-x_{235} & x_{135} & -x_{125} & 0 & x_{123}\\
	x_{234} & -x_{134} & x_{124} & -x_{123} & 0
	\end{bmatrix}.
\]
The coefficient of $\epsilon_6$ has the form
\[
	y + \frac{1}{4}\sum_{I \subset [6]} \pm x_I x_{[6]\setminus I}.
\]
From this computation of $d_4$ (and thus of $d_1$), the Genericity Conjecture predicts that Gorenstein ideals of codimension four on $n=6$ generators are hypersurface sections of Gorenstein ideals of codimension three on 5 generators. This is a well-known conjecture that has been established under some mild hypotheses; see \cite{HM85} and \cite{VV86}.

The ring $R$ and this resolution are $\mb{Z}^6$-graded; the multigrading is displayed in Figure~\ref{fig:length4-E6}.

\subsection{$E_7$}
We draw the Dynkin diagram as:
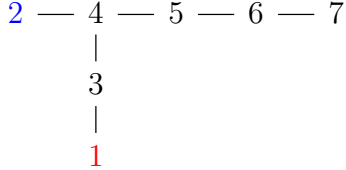
\[\begin{tikzcd}[column sep = small, row sep = small]
\colorC{2} \ar[r,dash] & 4 \ar[d,dash]\ar[r,dash] & 5 \ar[r,dash] & 6 \ar[r,dash] & 7\\
& 3 \ar[d,dash]\\
& \colorB{1}
\end{tikzcd}\]
Write $\mf{g}$ for the corresponding Lie algebra. We fix some subalgebras of $\mf{g}$ as follows:
\begin{itemize}
	\item $\mf{sl}(H)$ corresponding to the nodes $7,6,5,4,3$, where $H = \mb{C}^6$,
	\item $\mf{sl}(F)$ corresponding to the nodes $1,3,4,5,6,7$, where $F = \mb{C}^7$.
\end{itemize}
Note that the subalgebra corresponding to the nodes $7,6,5,4,3,2$ is $\mf{so}(H\oplus H^*)$.

The representation $V(\omega_1)$ is the adjoint. It has 5 graded components in the 1-grading and 5 in the 2-grading:
\begin{align*}
V(\omega_1) &= \CC \oplus \bigwedge^\mathrm{odd}H \oplus \cdots \oplus {\colorB{\boxed{\CC}}}\\
&= {\colorA{\boxed{F^*}}} \oplus {\bigwedge^4 F^*} \oplus \cdots \oplus {S_{2^6,1} F^*}
\end{align*}
The representation $V(\omega_7)$ has 3 graded components in the 1-grading and 4 in the 2-grading:
\begin{align*}
	V(\omega_7) &= (H \oplus H^*) \oplus \bigwedge^\mathrm{even}H \oplus {\colorB{\boxed{(H \oplus H^*)}}}\\
	&= {\colorA{\boxed{F}}} \oplus {\bigwedge^2 F^*} \oplus {\bigwedge^5 F^*} \oplus {S_{2,1^6} F^*}
\end{align*}
Take $R$ to be the same as in \S\ref{sec:length3} for the format $(1,5,7,3)$. We will describe the differential $d_2 = \exp_7$ in the resolution
\[
0 \to \colorB{\CC} \otimes R \xto{\exp_1} \colorA{F^*} \otimes R \xto{\exp_7^*} \colorB{(H \oplus H^*)}\otimes R \xto{\exp_7} \colorA{F}\otimes R \xto{\exp_1^*} \colorB{\CC} \otimes R.
\]
Many of the entries are shared with the differential $d_2$ for the $(1,6,7,2)$ resolution, since it is constructed using the same representation. The table of entries is
\[\setcounter{MaxMatrixCols}{20}
	\begin{bmatrix}
	35 & 35 & 35 & 35 & 35 & 35 & 11 & 11 & 11 & 11 & 11 & 11\\
	35 & 35 & 35 & 35 & 35 & 15 & 3 & 3 & 3 & 3 & 3 & 11\\
	35 & 35 & 35 & 35 & 15 & 35 & 3 & 3 & 3 & 3 & 11 & 3\\
	35 & 35 & 35 & 15 & 35 & 35 & 3 & 3 & 3 & 11 & 3 & 3\\
	35 & 35 & 15 & 35 & 35 & 35 & 3 & 3 & 11 & 3 & 3 & 3\\
	35 & 15 & 35 & 35 & 35 & 35 & 3 & 11 & 3 & 3 & 3 & 3\\
	15 & 35 & 35 & 35 & 35 & 35 & 11 & 3 & 3 & 3 & 3 & 3
	\end{bmatrix}.
\]
The 15-term entry in position $(7,1)$ and the 35-term entry in position $(6,1)$ are the same as in \eqref{eq:15-term} and \eqref{eq:35-term} respectively. The 11-term entry in position $(7,7)$ is
\begin{gather*}
	\frac{1}{2}\Big(
	x_{345}x_{126} - x_{245}x_{136} - x_{145}x_{236} + x_{235}x_{146} + x_{135}x_{246}\\
	- x_{125}x_{346} - x_{234}x_{156} - x_{134}x_{256} + x_{124}x_{356} - x_{123}x_{456}
	\Big) - y_{123456}
\end{gather*}
and the 3-term entry in position $(7,8)$ is
\[
	-x_{145}x_{136} + x_{135}x_{146} - x_{134}x_{156}.
\]

The multigrading on the resolution is displayed in Figure~\ref{fig:length4-E7}. 

\subsection{$E_8$}
We draw the Dynkin diagram as:
\[\begin{tikzcd}[column sep = small, row sep = small]
\colorC{2} \ar[r,dash] & 4 \ar[d,dash]\ar[r,dash] & 5 \ar[r,dash] & 6 \ar[r,dash] & 7 \ar[r,dash] & 8\\
& 3 \ar[d,dash]\\
& \colorB{1}
\end{tikzcd}\]
Write $\mf{g}$ for the corresponding Lie algebra. We fix some subalgebras of $\mf{g}$ as follows:
\begin{itemize}
	\item $\mf{sl}(H)$ corresponding to the nodes $8,7,6,5,4,3$, where $H = \mb{C}^7$,
	\item $\mf{sl}(F)$ corresponding to the nodes $1,3,4,5,6,7,8$, where $F = \mb{C}^8$.
\end{itemize}
Note that the subalgebra corresponding to the nodes $8,7,6,5,4,3,2$ is $\mf{so}(H\oplus H^*)$.

The representation $V(\omega_1)$ has 9 graded components in the 1-grading and 11 in the 2-grading:
\begin{align*}
V(\omega_1) &= \CC \oplus \bigwedge^\mathrm{odd}H \oplus \cdots \oplus {\colorB{\boxed{\CC}}}\\
&= {\colorA{\boxed{F^*}}} \oplus {\bigwedge^4 F^*} \oplus \cdots \oplus {S_{4^7,3} F^*}
\end{align*}
The representation $V(\omega_8)$ is the adjoint. It has 5 graded components in the 1-grading and 7 in the 2-grading:
\begin{align*}
V(\omega_8) &= (H \oplus H^*) \oplus \bigwedge^\mathrm{even}H \oplus (\CC\oplus \bigwedge^2(H\oplus H^*))\oplus \bigwedge^\mathrm{odd}H \oplus {\colorB{\boxed{(H \oplus H^*)}}}\\
&= {\colorA{\boxed{F}}} \oplus {\bigwedge^2 F^*} \oplus \cdots \oplus {S_{3,2^7} F^*}
\end{align*}
The construction yields a resolution
\[
0 \to \colorB{\CC} \otimes R \xto{\exp_1} \colorA{F^*} \otimes R \xto{\exp_8^*} \colorB{(H \oplus H^*)}\otimes R \xto{\exp_8} \colorA{F}\otimes R \xto{\exp_1^*} \colorB{\CC} \otimes R
\]
with multigrading displayed in Figure~\ref{fig:length4-E8}. We will not describe any of the differentials explicitly, though we remark that $\exp_8$ has been implemented in Macaulay2.

\begin{figure}[!h]
	\begin{align*}
	0 \to
	&\begin{matrix}
	R(-(\colorB{6}, & \colorA{7}, & 9, & 12, & 8, & 4))
	\end{matrix}
	\to\\[1em]\xrightarrow{\exp_1}
	\bigoplus&\begin{matrix}
	R(-(\colorB{4}, & \colorA{5}, & 6, & 8, & 5, & 2))\\
	R(-(\colorB{4}, & \colorA{5}, & 6, & 8, & 5, & 3))\\
	R(-(\colorB{4}, & \colorA{5}, & 6, & 8, & 6, & 3))\\
	R(-(\colorB{4}, & \colorA{5}, & 6, & 9, & 6, & 3))\\
	R(-(\colorB{4}, & \colorA{5}, & 7, & 9, & 6, & 3))\\
	R(-(\colorB{5}, & \colorA{5}, & 7, & 9, & 6, & 3))
	\end{matrix}
	\to\\[1em]\xrightarrow{\exp_6^*}
	\bigoplus&\begin{matrix}
	R(-(\colorB{3}, & \colorA{3}, & 4, & 5, & 3, & 1))\\
	R(-(\colorB{3}, & \colorA{3}, & 4, & 5, & 3, & 2))\\
	R(-(\colorB{3}, & \colorA{3}, & 4, & 5, & 4, & 2))\\
	R(-(\colorB{3}, & \colorA{3}, & 4, & 6, & 4, & 2))\\
	R(-(\colorB{3}, & \colorA{3}, & 5, & 6, & 4, & 2))\\
	R(-(\colorB{3}, & \colorA{4}, & 4, & 6, & 4, & 2))\\
	R(-(\colorB{3}, & \colorA{4}, & 5, & 6, & 4, & 2))\\
	R(-(\colorB{3}, & \colorA{4}, & 5, & 7, & 4, & 2))\\
	R(-(\colorB{3}, & \colorA{4}, & 5, & 7, & 5, & 2))\\
	R(-(\colorB{3}, & \colorA{4}, & 5, & 7, & 5, & 3))
	\end{matrix}
	\to\\[1em]\xrightarrow{\exp_6}
	\bigoplus&\begin{matrix}
	R(-(\colorB{1}, & \colorA{2}, & 2, & 3, & 2, & 1))\\
	R(-(\colorB{2}, & \colorA{2}, & 2, & 3, & 2, & 1))\\
	R(-(\colorB{2}, & \colorA{2}, & 3, & 3, & 2, & 1))\\
	R(-(\colorB{2}, & \colorA{2}, & 3, & 4, & 2, & 1))\\
	R(-(\colorB{2}, & \colorA{2}, & 3, & 4, & 3, & 1))\\
	R(-(\colorB{2}, & \colorA{2}, & 3, & 4, & 3, & 2))
	\end{matrix}
	\to\\\xrightarrow{\exp_1^*}
	&R.
	\end{align*}
	\caption{Multigrading on resolution constructed from $E_6$.}\label{fig:length4-E6}
\end{figure}
\begin{figure}[!h]
	\begin{align*}
	0 \to
	&\begin{matrix}
	R(-(\colorB{10}, & \colorA{13}, & 17, & 24, & 18, & 12, & 6))
	\end{matrix}
	\to\\[1em]\xrightarrow{\exp_1}
	\bigoplus&\begin{matrix}
	R(-(\colorB{7}, & \colorA{9}, & 12, & 17, & 13, & 9, & 5))\\
	R(-(\colorB{7}, & \colorA{9}, & 12, & 17, & 13, & 9, & 4))\\
	R(-(\colorB{7}, & \colorA{9}, & 12, & 17, & 13, & 8, & 4))\\
	R(-(\colorB{7}, & \colorA{9}, & 12, & 17, & 12, & 8, & 4))\\
	R(-(\colorB{7}, & \colorA{9}, & 12, & 16, & 12, & 8, & 4))\\
	R(-(\colorB{7}, & \colorA{9}, & 11, & 16, & 12, & 8, & 4))\\
	R(-(\colorB{6}, & \colorA{9}, & 11, & 16, & 12, & 8, & 4))\\
	\end{matrix}
	\to\\[1em]\xrightarrow{\exp_7^*}
	\bigoplus&\begin{matrix}
	R(-(\colorB{5}, & \colorA{6}, & 8, & 11, & 8, & 5, & 2))\\
	R(-(\colorB{5}, & \colorA{6}, & 8, & 11, & 8, & 5, & 3))\\
	R(-(\colorB{5}, & \colorA{6}, & 8, & 11, & 8, & 6, & 3))\\
	R(-(\colorB{5}, & \colorA{6}, & 8, & 11, & 9, & 6, & 3))\\
	R(-(\colorB{5}, & \colorA{6}, & 8, & 12, & 9, & 6, & 3))\\
	R(-(\colorB{5}, & \colorA{6}, & 9, & 12, & 9, & 6, & 3))\\
	R(-(\colorB{5}, & \colorA{7}, & 9, & 13, & 10, & 7, & 4))\\
	R(-(\colorB{5}, & \colorA{7}, & 9, & 13, & 10, & 7, & 3))\\
	R(-(\colorB{5}, & \colorA{7}, & 9, & 13, & 10, & 6, & 3))\\
	R(-(\colorB{5}, & \colorA{7}, & 9, & 13, & 9, & 6, & 3))\\
	R(-(\colorB{5}, & \colorA{7}, & 9, & 12, & 9, & 6, & 3))\\
	R(-(\colorB{5}, & \colorA{7}, & 8, & 12, & 9, & 6, & 3))
	\end{matrix}
	\to\\[1em]\xrightarrow{\exp_7}
	\bigoplus&\begin{matrix}
	R(-(\colorB{3}, & \colorA{4}, & 5, & 7, & 5, & 3, & 1))\\
	R(-(\colorB{3}, & \colorA{4}, & 5, & 7, & 5, & 3, & 2))\\
	R(-(\colorB{3}, & \colorA{4}, & 5, & 7, & 5, & 4, & 2))\\
	R(-(\colorB{3}, & \colorA{4}, & 5, & 7, & 6, & 4, & 2))\\
	R(-(\colorB{3}, & \colorA{4}, & 5, & 8, & 6, & 4, & 2))\\
	R(-(\colorB{3}, & \colorA{4}, & 6, & 8, & 6, & 4, & 2))\\
	R(-(\colorB{4}, & \colorA{4}, & 6, & 8, & 6, & 4, & 2))
	\end{matrix}
	\to\\\xrightarrow{\exp_1^*}
	&R.
	\end{align*}
	\caption{Multigrading on resolution constructed from $E_7$.}\label{fig:length4-E7}
\end{figure}
\begin{figure}[!h]
\begin{align*}
0 \to
&\begin{matrix}
R(-(\colorB{22}, & \colorA{31}, & 41, & 60, & 48, & 36, & 24, & 12))
\end{matrix}
\to\\[1em]\xrightarrow{\exp_1}
\bigoplus&\begin{matrix}
R(-(\colorB{15}, & \colorA{21}, & 28, & 41, & 33, & 25, & 17, & 9))\\
R(-(\colorB{15}, & \colorA{21}, & 28, & 41, & 33, & 25, & 17, & 8))\\
R(-(\colorB{15}, & \colorA{21}, & 28, & 41, & 33, & 25, & 16, & 8))\\
R(-(\colorB{15}, & \colorA{21}, & 28, & 41, & 33, & 24, & 16, & 8))\\
R(-(\colorB{15}, & \colorA{21}, & 28, & 41, & 32, & 24, & 16, & 8))\\
R(-(\colorB{15}, & \colorA{21}, & 28, & 40, & 32, & 24, & 16, & 8))\\
R(-(\colorB{15}, & \colorA{21}, & 27, & 40, & 32, & 24, & 16, & 8))\\
R(-(\colorB{14}, & \colorA{21}, & 27, & 40, & 32, & 24, & 16, & 8))
\end{matrix}
\to\\[1em]\xrightarrow{\exp_8^*}
\bigoplus&\begin{matrix}
R(-(\colorB{11}, & \colorA{15}, & 20, & 29, & 23, & 17, & 11, & 5))\\
R(-(\colorB{11}, & \colorA{15}, & 20, & 29, & 23, & 17, & 11, & 6))\\
R(-(\colorB{11}, & \colorA{15}, & 20, & 29, & 23, & 17, & 12, & 6))\\
R(-(\colorB{11}, & \colorA{15}, & 20, & 29, & 23, & 18, & 12, & 6))\\
R(-(\colorB{11}, & \colorA{15}, & 20, & 29, & 24, & 18, & 12, & 6))\\
R(-(\colorB{11}, & \colorA{15}, & 20, & 30, & 24, & 18, & 12, & 6))\\
R(-(\colorB{11}, & \colorA{15}, & 21, & 30, & 24, & 18, & 12, & 6))\\
R(-(\colorB{11}, & \colorA{16}, & 21, & 31, & 25, & 19, & 13, & 7))\\
R(-(\colorB{11}, & \colorA{16}, & 21, & 31, & 25, & 19, & 13, & 6))\\
R(-(\colorB{11}, & \colorA{16}, & 21, & 31, & 25, & 19, & 12, & 6))\\
R(-(\colorB{11}, & \colorA{16}, & 21, & 31, & 25, & 18, & 12, & 6))\\
R(-(\colorB{11}, & \colorA{16}, & 21, & 31, & 24, & 18, & 12, & 6))\\
R(-(\colorB{11}, & \colorA{16}, & 21, & 30, & 24, & 18, & 12, & 6))\\
R(-(\colorB{11}, & \colorA{16}, & 20, & 30, & 24, & 18, & 12, & 6))
\end{matrix}
\to\\[1em]\xrightarrow{\exp_8}
\bigoplus&\begin{matrix}
R(-(\colorB{7}, & \colorA{10}, & 13, & 19, & 15, & 11, & 7, & 3))\\
R(-(\colorB{7}, & \colorA{10}, & 13, & 19, & 15, & 11, & 7, & 4))\\
R(-(\colorB{7}, & \colorA{10}, & 13, & 19, & 15, & 11, & 8, & 4))\\
R(-(\colorB{7}, & \colorA{10}, & 13, & 19, & 15, & 12, & 8, & 4))\\
R(-(\colorB{7}, & \colorA{10}, & 13, & 19, & 16, & 12, & 8, & 4))\\
R(-(\colorB{7}, & \colorA{10}, & 13, & 20, & 16, & 12, & 8, & 4))\\
R(-(\colorB{7}, & \colorA{10}, & 14, & 20, & 16, & 12, & 8, & 4))\\
R(-(\colorB{8}, & \colorA{10}, & 14, & 20, & 16, & 12, & 8, & 4))
\end{matrix}
\to\\\xrightarrow{\exp_1^*}
&R.
\end{align*}
\caption{Multigrading on resolution constructed from $E_8$.}\label{fig:length4-E8}
\end{figure}

\clearpage
\section{Proof of acyclicity}\label{sec:proofres}
We now relate the complexes considered in this paper to Schubert varieties, and prove their acyclicity in the process.
\begin{thm}\label{thm:acyclic}
	The two families of complexes defined in \S\ref{sec:general-construction} are acyclic.
\end{thm}
The main tool from commutative algebra we will need is the Buchsbaum-Eisenbud acyclicity criterion:
\begin{thm}[\cite{BE73}]\label{thm:acyclicity-criterion}
	Let $R$ be a ring. A complex
	\[
	0 \to F_n \xto{d_n} F_{n-1} \xto{d_{n-1}} \cdots \xto{d_2} F_2 \xto{d_1} F_0
	\]
	of free $R$-modules is exact iff
	\[
	\rank F_k = \rank d_k + \rank d_{k+1}
	\]
	and
	\[
	\operatorname{depth} I(d_k) \geq k
	\]
	for $k = 1,\ldots,n$, where $I(d_k)$ is the ideal of $(\rank d_k)\times(\rank d_k)$ minors of $d_k$.
\end{thm}
Note that our complexes are all over polynomial rings, so depth is the same as codimension.

The rest of the proof comes from representation theory---namely, the study of Schubert varieties and Pl\"ucker coordinates. The needed background can be found in \cite{LR}. 
We will set things up to match \cite{SW21}, in preparation for \S\ref{sec:invariants}. Fix a Dynkin diagram $T$, and let $\mf{g}$ (resp. $G$) denote the associated complex simple Lie algebra (resp. simply-connected group). Choose a vertex $i \in T$ and let $P_i$ denote the maximal parabolic subgroup corresponding to the nonpositive part $\mf{p}_i \subset \mf{g}$ in the grading induced by the simple root $\alpha_i$ (henceforth abbreviated ``$i$-grading'').

Let $W$ denote the Weyl group of $T$ and $W_{P_i}\subset W$ the Weyl group of $T \setminus \{i\}$. Both the Schubert varieties $X_w \subset G/P_i$ as well as the (extremal) Pl\"ucker coordinates $p_w$ on $G/P_i$ are indexed by elements $w$ of $W/W_{P_i}$. If $\tau, w \in W$ are shortest length representatives of their respective cosets in $W/W_{P_i}$, the function $p_\tau$ vanishes on $X_w$ if and only if $\tau \not\leq w$ in the Bruhat order. Moreover, $\{p_\tau = 0 : \tau \not\leq w\}$ cuts out $X_w$ set-theoretically. We will write $X^w = w_0 X_{w_0 w}$ for opposite Schubert varieties, where $w_0 \in W$ is the longest element. If $w$ is a shortest length representative of its coset, then $\codim X^w = \ell(w)$.

Let $V = V(\omega_i)^*$ be the representation with lowest weight $-\omega_i$, and let $V^{\mathrm{bot}}$ denote its one-dimensional bottom component in the $i$-grading, spanned by a lowest weight vector. The Pl\"ucker embedding of $G/P_i$ into $\mb{P}(V)$ is given by the formula
\begin{equation}\label{eq:flag-embed}
	\begin{split}
	G/P_i &\to \mb{P}(V)\\
	[g] &\mapsto [V^\mathrm{bot} \hookrightarrow V \xto{g\cdot} V]
	\end{split}
\end{equation}
The Pl\"ucker coordinates $p_w$ are then just the homogeneous coordinates on $\mb{P}(V)$ corresponding to the extremal weights $w\cdot (-\omega_i)$ in $V$.

Let $\mf{n}_i \subset \mf{g}$ be the negative part of $\mf{g}$ in the $i$-grading, and let $N_i$ be the corresponding unipotent subgroup. Then $N_i w_0 \subset G/P_i$ is the open Schubert cell $C_{w_0}$, which is given by the non-vanishing of the Pl\"ucker coordinate $p_{w_0}$, which we will normalize to equal 1. Thus we can parametrize this patch via
\begin{align*}
	\mf{n}_i &\to \mb{P}(V)\\
	X &\mapsto [\mb{C} \hookrightarrow V \xto{\exp(X)\cdot} V]
\end{align*}
where $\mb{C} \hookrightarrow V$ is the \emph{highest} weight line, i.e. the one-dimensional top component in the grading induced by the node dual to $i$ on $T$.

As before, write $R_i$ for the polynomial ring $\Sym \mf{n}_i^*$. The big open cell is thus $\Spec \Sym \mf{n}_i^*$.

\begin{example}
	Before we show how this machinery can be used to prove Theorem~\ref{thm:acyclic}, we first illustrate it concretely in the case of type $D_n$ for $n$ even. We will take $i = n$ and suppress the subscript $i$ in the following.
	
	The quotient $G/P$ is the orthogonal Grassmannian $OG(n,2n)$, which can be viewed inside of the classical Grassmannian $\Gr(n,2n)$ via the map
	\begin{align*}
		G/P &\to \Gr(n,2n)\\
		[g] &\mapsto [V(\omega_1)^\mathrm{bot} \hookrightarrow V(\omega_1) \xto {g\cdot} V(\omega_1)]
	\end{align*}
	where, letting $SL(F)$ denote the subgroup corresponding to the vertices $1,\ldots,n-1$ on the Dynkin diagram, the standard representation $V(\omega_1)$ is $F^* \oplus F$ with the evident quadratic form $Q(\varphi,x) = \varphi(x)$.
	
	Let $B$ be a generic element of $\mf{n}=\bigwedge^2 F^*$, viewed as a skew-symmetric matrix $F \to F^*$ with entries in $R = \Sym \mf{n}^*$. The open cell is given parametrically by the block matrix $\begin{bmatrix}
	I_n & B
	\end{bmatrix}^\top$. The skew-symmetry of $B$ reflects that the column span in $F^* \oplus F$ is an isotropic subspace for the quadratic form $Q$.
	
	However, the Pl\"ucker coordinates on $G/P$ do not appear as minors of this matrix---rather, their squares do. To see this, recall that the Pl\"ucker coordinates come from the embedding of $G/P$ into $\mb{P}(V(\omega_n))$. The map $V(\omega_n)^{\mathrm{top}} \hookrightarrow V(\omega_n) \xto{\exp B} V(\omega_n)$ parametrizing the open Schubert cell is
	\[
		\begin{bmatrix}
		Q_0 & Q_1 & \cdots & Q_{n/2-1} & Q_{n/2}
		\end{bmatrix}^\top
	\]
	where $Q_j$ consists of the $(2j)\times(2j)$ Pfaffians of $B$. In particular, $p_{w_0} = Q_0 = 1$ and $p_1 = Q_{n/2}$ is the Pfaffian of $B$.
	
	On the other hand, the minors of $\begin{bmatrix}
	I_n & B
	\end{bmatrix}^\top$ are coordinates on $\mb{P}(\bigwedge^n V(\omega_1))$, inside of which we have $\mb{P}(V(2\omega_n))$. This can be seen by a weight calculation: the weights in $F \subset V(\omega_1)$ can be obtained by reflecting $\omega_1$ sequentially by $s_1,\ldots,s_{n-1}$, and their sum is
	\begin{gather*}
		\omega_1 + (\omega_2 - \omega_1) + (\omega_3 - \omega_2) + \cdots + (\omega_{n-2} - \omega_{n-3})\\
		+ (\omega_n + \omega_{n-1} - \omega_{n-2}) + (\omega_n - \omega_{n-1}) = 2\omega_n.
	\end{gather*}
	Thus $\bigwedge^n F \subset \bigwedge^n V(\omega_1)$ is the highest weight part of $V(2\omega_n) = \bigwedge^n F^* \oplus \cdots \oplus \bigwedge^n F \subset \bigwedge^n V(\omega_1)$. If $p_w$ is the coordinate for the extremal weight $\omega$ in $V(\omega_n)$, then $p_w^2$ will be the coordinate for the extremal weight $2\omega$ in $V(2\omega_n) \subset \bigwedge^n V(\omega_1)$; for example $p_1^2 = \det B$.
\end{example}
This example illustrates how we can embed $G/P$ into classical Grassmannians and (powers of) the Pl\"ucker coordinates for $G/P$ appear among the classical Pl\"ucker coordinates for the Grassmannian. We will exploit this to prove Theorem~\ref{thm:acyclic}.

For all the complexes we've constructed, the proofs of acyclicity are almost entirely the same, so we will just treat one specific case from each family to avoid notational complications. Minor adjustments are only required for Dynkin diagrams with exceptional duality as per Remark~\ref{rem:dual-node}. We will not discuss this adjustment, because it is only necessary for type $A_n$ (uninteresting), $D_n$ (known), and $E_6$ (simple enough to verify acyclicity in Macaulay2).

\begin{proof}[Proof of acyclicity (for codimension three examples)]
For length three, we will prove acyclicity of the $(1,7,8,2)$ complexes $\mb{F}_2, \mb{F}_3$ constructed from $E_8$. By Lemma~\ref{lem:dual-format-equiv}, it is sufficient to handle one of the two, so we will just consider $\mb{F}_2$. The same ideas can be used to directly prove acyclicity of $\mb{F}_3$; however treating $\mb{F}_2$ will be useful for \S\ref{sec:invariants}. So let $G = E_8$ and $i=2$. We will suppress the subscript $i$ in what follows. The quotient $W/W_P$ is in bijection with the $W$-orbit of the weight
\[\renewcommand*{\arraycolsep}{2pt}
	\omega_2 = \begin{matrix}
	1 & 0 & 0 & 0 & 0 & 0\\
	& 0\\
	& 0
	\end{matrix}.
\]
\begingroup
\begin{figure}[!h]
\renewcommand*{\arraycolsep}{2pt}
\[\begin{tikzcd}[ampersand replacement=\&]
\&\&\begin{matrix}
1 & 0 & 0 & 0 & 0 & 0\\
& 0\\
& 0
\end{matrix} \ar[d,dash,"s_2"]\\
\&\&\begin{matrix}
-1 & 1 & 0 & 0 & 0 & 0\\
& 0\\
& 0
\end{matrix} \ar[d,dash,"s_4"]\\
\&\&\begin{matrix}
0 & -1 & 1 & 0 & 0 & 0\\
& 1\\
& 0
\end{matrix}\ar[dl,dash,"s_3",swap] \ar[dr,dash,"s_5"]\\
\&\begin{matrix}
0 & 0 & 1 & 0 & 0 & 0\\
& -1\\
& 1
\end{matrix} \ar[dl,dash,"s_1",swap] \ar[dr,dash,"s_5"]\&\& \begin{matrix}
0 & 0 & -1 & 1 & 0 & 0\\
& 1\\
& 0
\end{matrix} \ar[dl,dash,"s_3",swap] \ar[dr,dash,"s_6"]\\
\cdots\cdots\cdots \&\& \cdots\cdots\cdots \&\& \cdots\cdots\cdots
\end{tikzcd}\]
\caption{$W/W_{P_2}$ for $E_8$.}\label{fig:Bruhat3}
\end{figure}
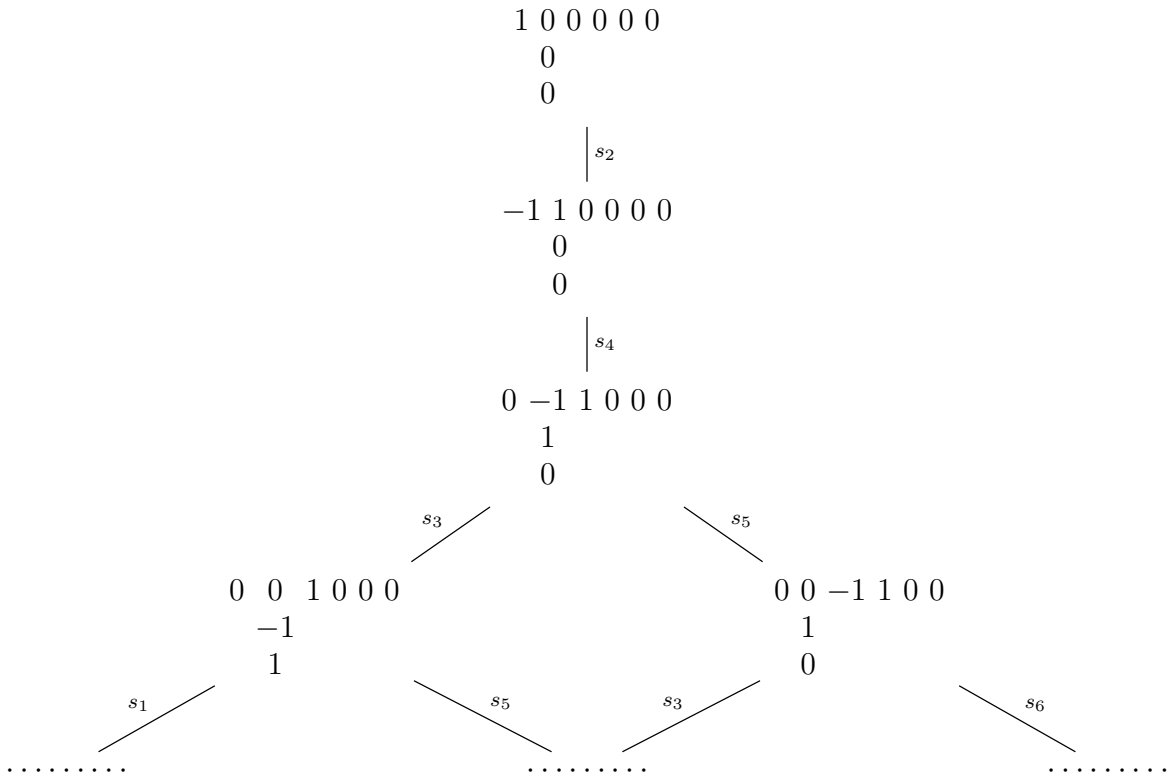
\endgroup
Acting by elements $w \in W$ of minimal length $\ell$ in their coset, we get the Bruhat graph for $W/W_P$ in Figure~\ref{fig:Bruhat3}, where $s_i$ denotes reflection at the vertex $i$. The Pl\"ucker coordinates $p_1, p_{s_2}$ set-theoretically cut out the codimension two (opposite) Schubert variety $X^{s_4 s_2}$. The coordinate $p_{s_4s_2}$ does not vanish on $X^{s_4 s_2}$, thus the vanishing locus of $(p_1, p_{s_2}, p_{s_4s_2})$ has codimension three. (In fact it is the union of the two codimension three Schubert varieties $X^{s_3 s_4 s_2}, X^{s_5 s_4 s_2}$.)

Now we will locate powers of these three Pl\"ucker coordinates in each ideal $I(d_k)$, which will show $\codim I(d_k) = 3$. For the differential $d_3$, let $V = V(\omega_1)$ and consider the embedding of $N$ into $\Gr(8,V)$ via \eqref{eq:flag-embed}. The representation $V$ has 11 graded components in the 2-grading:
\[
	V = F_2 \oplus \bigwedge^4 F_2 \oplus \cdots \oplus S_{4^7,3} F_2
\]
(where $S_{4^7,3}F_2 \cong F_2^*$) so this embedding is given parametrically by a $3875 \times 8$ block matrix, transposed for compactness,
\[
	\begin{bmatrix}
	I & Q_1 & \cdots & Q_{9} & Q_{10}
	\end{bmatrix}^\top
\]
where the entries of $Q_j$ have degree $j$. In particular $Q_{10}$ is the part mapping to the bottom graded component, and $Q_0$ is just the identity matrix.

The sum of the weights appearing in the top component is
\begin{gather*}
	\omega_1 + (\omega_3 - \omega_1) + (\omega_4 - \omega_3)
	+ (\omega_2 + \omega_5 - \omega_4) + (\omega_2 + \omega_6 - \omega_5)\\
	+ (\omega_2 + \omega_7 - \omega_6) + (\omega_2 + \omega_8 - \omega_7)
	+ (\omega_2 - \omega_8) = 5\omega_2
\end{gather*}
exhibiting $V(5\omega_2) = \bigwedge^8 F_2 \oplus \cdots \oplus \bigwedge^8 F_2^*$ inside $\bigwedge^8 V(\omega_1)$. So, after normalizing $p_{w_0} = 1$ on $N$, we will find $p_w^5$ as a minor of the $3875 \times 8$ matrix for each Pl\"ucker coordinate $p_w$.

Observe that $d_3^*$ is the $2 \times 8$ submatrix mapping to the the bottom part $F_3 \subset V$ in the 3-grading. We have $p_1^5 = \det Q_{10}$, so the minor computing $p_1^5$ evidently involves the two columns of $Q_{10}$ comprising $d_3$. The crucial observation is that the minor $p_w^5$ will also involve these two columns if $w$ does not contain the reflection $s_3$. This is because each $s_i \in W$ other than $s_3$ is represented by an element of $SL(F_1) \times SL(F_3)$. Thus we conclude the minors $p_1^5$, $p_{s_2}^5$, and $p_{s_4s_2}^5$ are in $I_2(d_3)$ (e.g. by Laplace expansion).

For $d_2$, we instead take $V = V(\omega_8)$. We get an embedding of $N$ into $\Gr(8,V)$ by a $248 \times 8$ matrix
\[
	\begin{bmatrix}
	I & Q_1 & \cdots & Q_5 & Q_6
	\end{bmatrix}^\top.
\]
Now $p_1^3$ is $\det Q_6$. Six columns of $Q_6$ come from the bottom 3-graded piece $F_1^*\subset V$, and thus the same is true for the minors $p_{s_2}^3$ and $p_{s_4s_2}^3$ (again because we do not use the reflection $s_3$). Hence all three minors are in $I_6(d_2)$.

Finally, for $d_1$, we take $V = V(\omega_2)$. We get an embedding of $N$ into $\mb{P}(V)$ by a $147250 \times 1$ matrix
\[
\begin{bmatrix}
1 & Q_1 & \cdots & Q_{15} & Q_{16}
\end{bmatrix}^\top.
\]
In this embedding, $p_1$ is just the entry $Q_{16}$. The differential $d_1^*$ is the part of this matrix mapping to the bottom 3-graded piece $F_1 \subset V$. By the same reasoning as before, we deduce that $p_{s_2}$ and $p_{s_4s_2}$ are also entries of $d_1$.
\end{proof}

\begin{proof}[Proof of acyclicity (for codimension four Gorenstein examples)]
\begingroup
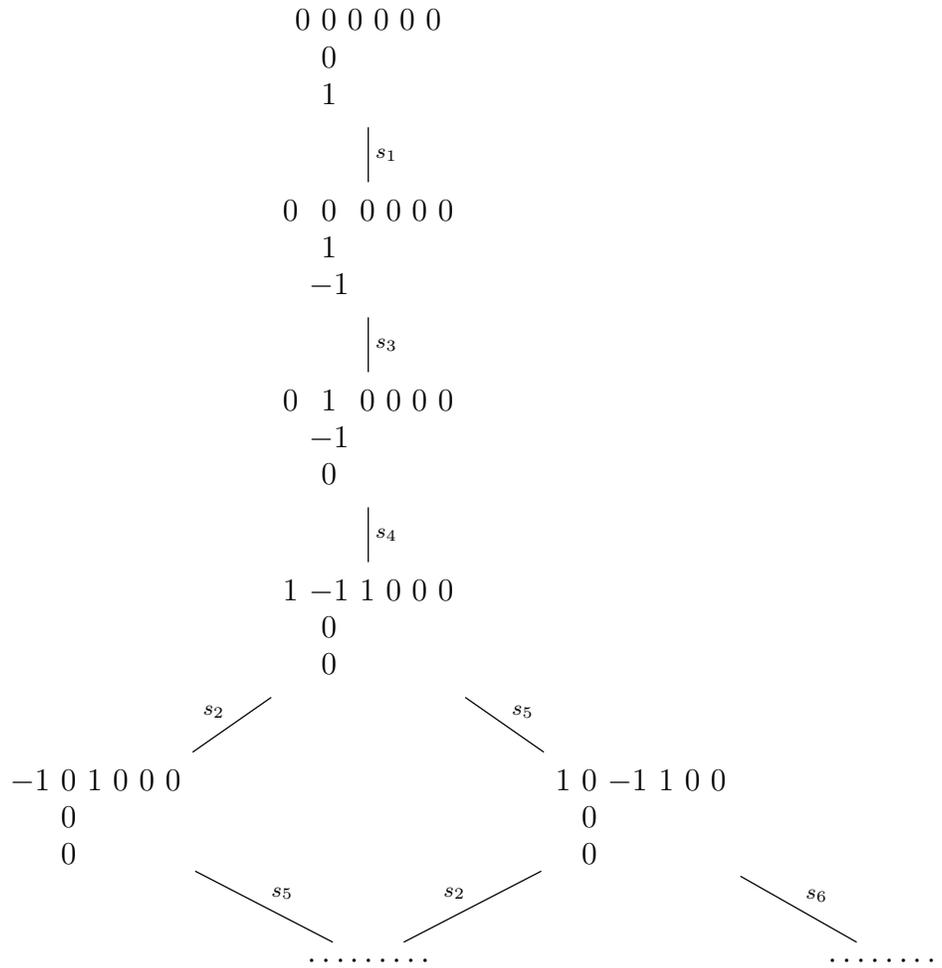
\begin{figure}[!h]
	\renewcommand*{\arraycolsep}{2pt}
	\[\begin{tikzcd}[ampersand replacement=\&]
	\&\begin{matrix}
	0 & 0 & 0 & 0 & 0 & 0\\
	& 0\\
	& 1
	\end{matrix} \ar[d,dash,"s_1"]\\
	\&\begin{matrix}
	0 & 0 & 0 & 0 & 0 & 0\\
	& 1\\
	& -1
	\end{matrix} \ar[d,dash,"s_3"]\\
	\&\begin{matrix}
	0 & 1 & 0 & 0 & 0 & 0\\
	& -1 \\
	& 0
	\end{matrix} \ar[d,dash,"s_4"]\\
	\&\begin{matrix}
	1 & -1 & 1 & 0 & 0 & 0\\
	& 0\\
	& 0
	\end{matrix}\ar[dl,dash,"s_2",swap] \ar[dr,dash,"s_5"]\\
	\begin{matrix}
	-1 & 0 & 1 & 0 & 0 & 0\\
	& 0\\
	& 0
	\end{matrix} \ar[dr,dash,"s_5"]\&\& \begin{matrix}
	1 & 0 & -1 & 1 & 0 & 0\\
	& 0\\
	& 0
	\end{matrix} \ar[dl,dash,"s_2",swap] \ar[dr,dash,"s_6"]\\
	\& \cdots\cdots\cdots \&\& \cdots\cdots\cdots
	\end{tikzcd}\]
	\caption{$W/W_{P_1}$ for $E_8$.}\label{fig:Bruhat4}
\end{figure}
\endgroup
We will consider the examples constructed from $E_8$. By Lemma~\ref{lem:dual-format-equiv}, it will suffice to treat $\mb{F}_1$. The proof method is the same as for length three: it suffices to find enough powers of Pl\"ucker coordinates in the ideals $I(d_k)$ to deduce they have codimension four. Note that the complex is self-dual, so it suffices to consider $I(d_1)$ and $(d_2)$.

This time we take $i=1$, and the top of the Bruhat graph for $W/W_P$ is shown in Figure~\ref{fig:Bruhat4}.

Mimicking the argument for length three, one can find powers of the top four Pl\"ucker coordinates $p_1, p_{s_1}, p_{s_3 s_1}, p_{s_4 s_3 s_1}$ in each ideal. One considers the embedding of $G/P$ into $\Gr(14,V(\omega_8))$ to show this for $I_7(d_2)$, and the embedding into $\mb{P}(V(\omega_1))$ to show this for $I_1(d_1)$. The rest of the proof is exactly the same so we omit it.
\end{proof}
\clearpage

\section{Invariants}\label{sec:invariants}
For the formats associated to Dynkin diagrams $E_n$ ($n=6,7,8$), the proof of Theorem~\ref{thm:acyclic} actually shows that the complexes $\mb{F}_2$ resolve the coordinate rings of the (restricted) Schubert varieties $Y^{\sigma_3}, Y^{\sigma_3'}$ considered in \cite{SW21}. Taking $(1,7,8,2)$ as in the proof for example, we find that the differential $d_1$ of $\mb{F}_2$ exactly consists of the Pl\"ucker coordinates $p_1,p_{s_2},p_{s_4s_2},\ldots,p_{s_8s_7s_6s_5s_4s_2}$, which generate the ideal of $Y^{\sigma_3}= N \cap X^{\sigma_3}$ in $N$, where $\sigma_3 = s_3 s_4 s_2$ \cite[\S5]{SW21}. In that paper, linkage was used to deduce the graded Betti numbers of each resolution, noting that the codimension forces various equations upon them---namely, $\sum (-1)^i t^{b_{i,j}}$ must be divisible by $(1-t)^3$. Now we have given explicit descriptions of the resolutions, and indeed the $\mb{Z}^n$-multigraded Betti numbers displayed in \S\ref{sec:length3} recover those stated in \cite{SW21} when coarsened to the 2-grading.

The linear sections of $Y^{\sigma_3}, Y^{\sigma_3'}$ obtained by setting all variables in $\mf{n}^*_{\geq 2}$ to zero are also studied in \cite{SW21}. We recast this in terms of our construction and expand it to include the resolutions of length four:
\begin{prop}\label{prop:linear-section}
	Let $i=2$ (resp. $i=1$). Let $\mf{n}_k$ denote the degree $-k$ part of $\mf{g}$ in the $i$-grading, $\mf{n} = \bigoplus_{k\geq 1} \mf{n}_k$, and $R = \Sym \mf{n}^*$. Consider the map $\phi\colon R \to R'\coloneqq \Sym \mf{n}_1^*$ sending all variables $\mf{n}^*_{\geq 2}$ to zero. The length three (resp. four) complexes $\mb{F}_i \otimes R'$ for $T = E_6, E_7, E_8$ (resp. $T = E_7, E_8$) resolve ideals in $R'$ generated by a $(T\setminus \{i\})$-invariant $\Delta$ together with some of its partial derivatives.
\end{prop}

We give a sketch of a computer-assisted proof of this statement. A conceptual proof would be desirable, but we remark that Proposition~\ref{prop:linear-section} is not essential to the broader goal of understanding the structure theory of free resolutions. These linear sections are attractive because of their particularly simple descriptions, but we do not expect them to be generic in the sense of Conjectures~\ref{conj:gen3} and \ref{conj:gen4}.

\begin{proof}
We will demonstrate the proposition for the $E_8$ format $(1,7,8,2)$, following the notation from the proof of Theorem~\ref{thm:acyclic}. We will use $(-)'$ to denote $- \otimes R'$ or $\phi(-)$. The differentials in $\mb{F}_2$ were constructed using the exponential action of $\mf{n}$, so those of $\mb{F}_2'$ can be obtained by taking $\mf{n}_1$ in place of $\mf{n}$.

In particular, the differential $d_1'$ comes from the exponential action of $\mf{n}_1$ on $V(\omega_2)$. This representation has a one-dimensional top component $\mb{C}$ in the 2-grading. Since it is self-dual, the bottom component is $\mb{C}$ as well.

Let $L'\in \mf{n}_1 \otimes R'$ represent a generic element of $\mf{n}_1$. The entry $p_1'$ of $d_1'$ comes from the part of $\exp L'$ mapping from the top component of $V(\omega_2)$ to the bottom. This is an $\mf{sl}(F_2)$-equivariant map $\mb{C} \to R'_{16}$, so if $p_1'$ is nonzero then it is an $\mf{sl}(F_2)$-invariant of degree 16 on $\mf{n}_1 = \bigwedge^3 F_2^*$.

Assume for the moment that $\Delta \coloneqq p_1' \neq 0$. Its partial derivatives with respect to the variables in $R'$ would span a copy of the representation $\bigwedge^3 F_2^*$ in $R'_{15}$. On the other hand, the part of $\exp L'$ mapping from the second-highest component of $V(\omega_2)$ to the bottom component would also span such a representation. It can be verified by computer\footnote{We owe thanks to Witek Kra\'skiewicz for carrying out this calculation.} that the representation $\bigwedge^3 F_2^*$ appears exactly once inside of $S_{15}(\bigwedge^3 F_2)$, so these must coincide. 
By matching the weights, we see that explicitly
\[
	p_\mathrm{bot}(\exp L')(\epsilon_i \wedge \epsilon_j \wedge \epsilon_k) = \frac{\partial \Delta}{\partial x_{ijk}}
\]
for $\epsilon_i \wedge \epsilon_j \wedge \epsilon_k$ in the second-highest component $\bigwedge^3 F_2^*\subset V(\omega_2)$, and $x_{ijk} \in R'_1 = \bigwedge^3 F_2^*$ the corresponding variable. Here $p_\mathrm{bot}$ is projection of $V(\omega_2)$ onto its bottom component.

For $(1,7,8,2)$, if $e_1,\ldots,e_8$ is a standard basis (c.f. Definition~\ref{def:std-basis}) for $F_2$, the differential $d_1'$ of $\mb{F}'_2$ is therefore
\[
	\begin{bmatrix}
	\Delta & \dfrac{\partial\Delta}{\partial x_{178}} & \cdots & \dfrac{\partial\Delta}{\partial x_{678}}
	\end{bmatrix}.
\]
As an aside, one could obtain $d_1'$ for $(1,5,8,4)$ by taking the partial derivatives with respect to $x_{678}, x_{578}, x_{568}, x_{567}$ instead. Observe that $p_1' = \Delta$, $p_{s_2}' = \partial \Delta / \partial x_{678}$, and $p_{s_4 s_2}' = \partial \Delta / \partial x_{567}$ are shared.

It remains to check $p_1' \neq 0$ and that $\mb{F}_2'$ is acyclic. Both statements would follow from knowing that $p_1',p_{s_2}',p_{s_4 s_2}'$ is a regular sequence in $R'$. Let $X \subset \Spec R'$ be the subscheme cut out by these functions. If $H$ is any linear subspace through the origin, then $H$ meets every component of $X$ because $X$ is conical. Since $\codim(X \cap H, H) \leq \codim(X, \Spec R')$, it is sufficient to verify that $p_1',p_{s_2}',p_{s_4 s_2}'$ is a regular sequence after substituting each variable in $R'$ to a random linear form in $S = \mb{C}[x,y,z]$ (i.e. taking $H$ to be a 3-plane). This has been checked by computer\footnote{For $E_8$, the representation $V(\omega_2)$ is too large to handle by computer, but the images of $p_1',p_{s_2}',p_{s_4 s_2}'$ in $S$ can be computed using the adjoint $V(\omega_8)$ instead. Namely, letting $L'' = L' \otimes S$, the images of $p_1'^3, p_{s_2}'^3, p_{s_4 s_2}'^3$ in $S$ appear as $8\times 8$ minors of $\exp L''$ on $V(\omega_8)$.}.
\end{proof}
For length four, one takes $V(\omega_1)$ instead of $V(\omega_2)$, and $\mf{n}_1$ is a half-spinor representation of $\mf{so}(H\oplus H^*)$, but the ideas are otherwise the same. This also explains why $E_6$ is omitted from Proposition~\ref{prop:linear-section}: the representation $V(\omega_1)$ is not self-dual in this case.

We summarize the invariants which appear in this process. They are thoroughly detailed in \cite{SK77}.
\begin{itemize}
	\item For the $(1,5,6,2)$ resolution constructed from $E_6/P_2$, $\Delta$ is an $\mf{sl}_6$-invariant of degree 4 on $\bigwedge^3 \CC^6$. This is listed on the first table in \cite[\S7]{SK77} as item (5).
	\item For $(1,6,7,2)$ and $(1,5,7,3)$ from $E_7/P_2$, an $\mf{sl}_7$-invariant of degree 7 on $\bigwedge^3 \CC^7$, item (6).
	\item For $(1,7,8,2)$ and $(1,5,8,4)$ from $E_8/P_2$, an $\mf{sl}_8$-invariant of degree 16 on $\bigwedge^3 \CC^8$, item (7).
	\item For $(1,7,12,7,1)$ from $E_7/P_1$, an $\mf{so}_{12}$-invariant of degree 4 on a half-spinor representation, item (23).
	\item For $(1,8,14,8,1)$ from $E_8/P_1$, an $\mf{so}_{14}$-invariant of degree 8 on a half-spinor representation, item (24).
\end{itemize}

\section{Acknowledgments}
The second author is supported by the grant MAESTRO NCN -
UMO-2019/34/A/ST1/00263 - Research in Commutative Algebra and Representation Theory, NAWA POWROTY - PPN/PPO/2018/1/00013/U/00001 - Applications of Lie algebras
to Commutative Algebra, and OPUS grant National Science Centre, Poland grant UMO-2018/29/BST1/01290. The authors would like to thank Ela Celikbas, Lars Christensen, David Eisenbud, Sara Angela Filippini, Lorenzo Guerrieri, Witek Kra\'skiewicz, Jai Laxmi, Steven Sam, Jacinta Torres, and Oana Veliche for helpful conversations about the contents of this paper and related topics.

\end{document}